\documentclass[12pt]{article}

\usepackage[utf8]{inputenc}
\usepackage[T1]{fontenc}
\usepackage[margin=2.5cm]{geometry}
\usepackage{authblk}

\usepackage[round]{natbib}
\usepackage{doi}

\usepackage{xcolor}
\definecolor{myteal}{HTML}{004b6f}
\usepackage{hyperref}
\hypersetup{
    colorlinks=true,
    linkcolor=myteal,
    citecolor=myteal,
    urlcolor=myteal,
}
\usepackage[setbuildercolon, altopenint, noslant]{matmatix}
\usepackage{amsthm, thmtools}
\usepackage{enumitem}

\usepackage{cleveref}

\newtheorem{theorem}{Theorem}[section]
\newtheorem{proposition}[theorem]{Proposition}
\newtheorem{corollary}[theorem]{Corollary}
\newtheorem{lemma}[theorem]{Lemma}
\newtheorem{remark}[theorem]{Remark}
\newtheorem{definition}[theorem]{Definition}

\DeclareMathOperator{\Id}{Id}
\DeclareMathOperator{\Ai}{\mathcal{A}}
\DeclareMathOperator{\diff}{\mathrm{d}\!}
\DeclareMathOperator{\e}{\mathbf{e}}

\newcommand{\hatA}{\,\hat{\!\Ai\,}{\!}^n_t}
\DeclarePairedDelimiter\abs{\lvert}{\rvert}%

\usepackage{amssymb,tikz}
\newcommand{\mysetminusD}{\hbox{\tikz{\draw[line width=0.6pt,line cap=round] (3pt,0) -- (0,6pt);}}}
\newcommand{\mysetminusT}{\mysetminusD}
\newcommand{\mysetminusS}{\hbox{\tikz{\draw[line width=0.45pt,line cap=round] (2pt,0) -- (0,4pt);}}}
\newcommand{\mysetminusSS}{\hbox{\tikz{\draw[line width=0.4pt,line cap=round] (1.5pt,0) -- (0,3pt);}}}
\renewcommand{\setminus}{\mathbin{\mathchoice{\mysetminusD}{\mysetminusT}{\mysetminusS}{\mysetminusSS}}}

\author[1,2]{Félix Foutel-Rodier}
\author[1,2]{Amaury Lambert}
\author[1,2]{Emmanuel Schertzer}

\affil[1]{Laboratoire de Probabilités, Statistique et Modélisation
(LPSM), Sorbonne Université,\newline CNRS UMR 8001, Paris, France}
\affil[2]{Centre Interdisciplinaire de Recherche en Biologie (CIRB),
Collège de France,\newline PSL Research University, CNRS UMR 7241, Paris, France}

\title{Kingman's coalescent with erosion}

\begin{document}

\maketitle

\begin{abstract}
    Consider the Markov process taking values in the partitions of
    $\N$ such that each pair of blocks merges at rate one,
    and each integer is eroded, i.e., becomes a singleton block, at rate
    $d$. This is a special case of exchangeable
    fragmentation-coalescence process, called Kingman's coalescent with
    erosion. We provide a new construction of the stationary distribution
    of this process as a sample from a standard flow of bridges. This
    allows us to give a representation of the asymptotic frequencies of
    this stationary distribution in terms of a sequence of hierarchically
    independent diffusions. Moreover, we introduce a new process called
    Kingman's coalescent with immigration, where pairs of blocks coalesce
    at rate one, and new blocks of size one immigrate at rate $d$. By
    coupling Kingman's coalescents with erosion and with immigration, we
    are able to show that the size of a block chosen uniformly at random
    from the stationary distribution of the restriction of Kingman's coalescent with erosion
    to $\Set{1, \dots, n}$ converges to the total progeny of a critical binary branching
    process.
\end{abstract}

\section{Introduction}

    \subsection{Motivation}

In evolutionary biology, speciation refers to the event when two
populations from the same species lose the ability to exchange genetic
material, e.g.\ due to the formation of a new geographic barrier or 
accumulation of genetic incompatibilities. Even
if speciation is usually thought of as irreversible, 
related species can often still exchange genetic material through 
exceptional hybridization, migration events or sudden collapse of a
geographic barrier~\citep{roux_shedding_2016}.
This can lead to the transmission of chunks of DNA between
different species, a phenomenon
known as introgression, which is currently considered as a major evolutionary
force shaping the genomes of groups of related species~\citep{mallet_how_2016}.

Our study of Kingman's coalescent with erosion was first motivated by the
following simple model of speciation incorporating rare migration events,
depicted in \Cref{fig:model}. Consider a set of $N$ monomorphic species,
each harboring a genome of $n$ genes indexed by $\Set{1, \dots, n}$. We
model speciation by assuming that the dynamics of the species is
described by a Moran model: at rate one for each pair of species $(s_1, s_2)$, 
species $s_2$ dies, $s_1$ gives birth to a new species, replicates
its genome and sends it into the daughter species. We also model
introgression by assuming that at rate $d$ for each gene $g \in \Set{1, \dots, n}$ 
and each pair of species $(s_1, s_2)$, $g$ is replicated, the
new copy of $g$ is sent from $s_1$ to $s_2$ and replaces its homolog in
$s_2$.  This assumption is justified by the following view in terms of
individual migrants. Each time a migrant goes from $s_1$ to $s_2$, if
recombination is sufficiently strong, its genome rapidly gets washed out
by that of the resident species due to the frequent backcrosses (crosses
between descendants of the migrant and local residents) so that at most
one gene among $n$ reaches fixation. 
\begin{figure}
    \center
    \includegraphics{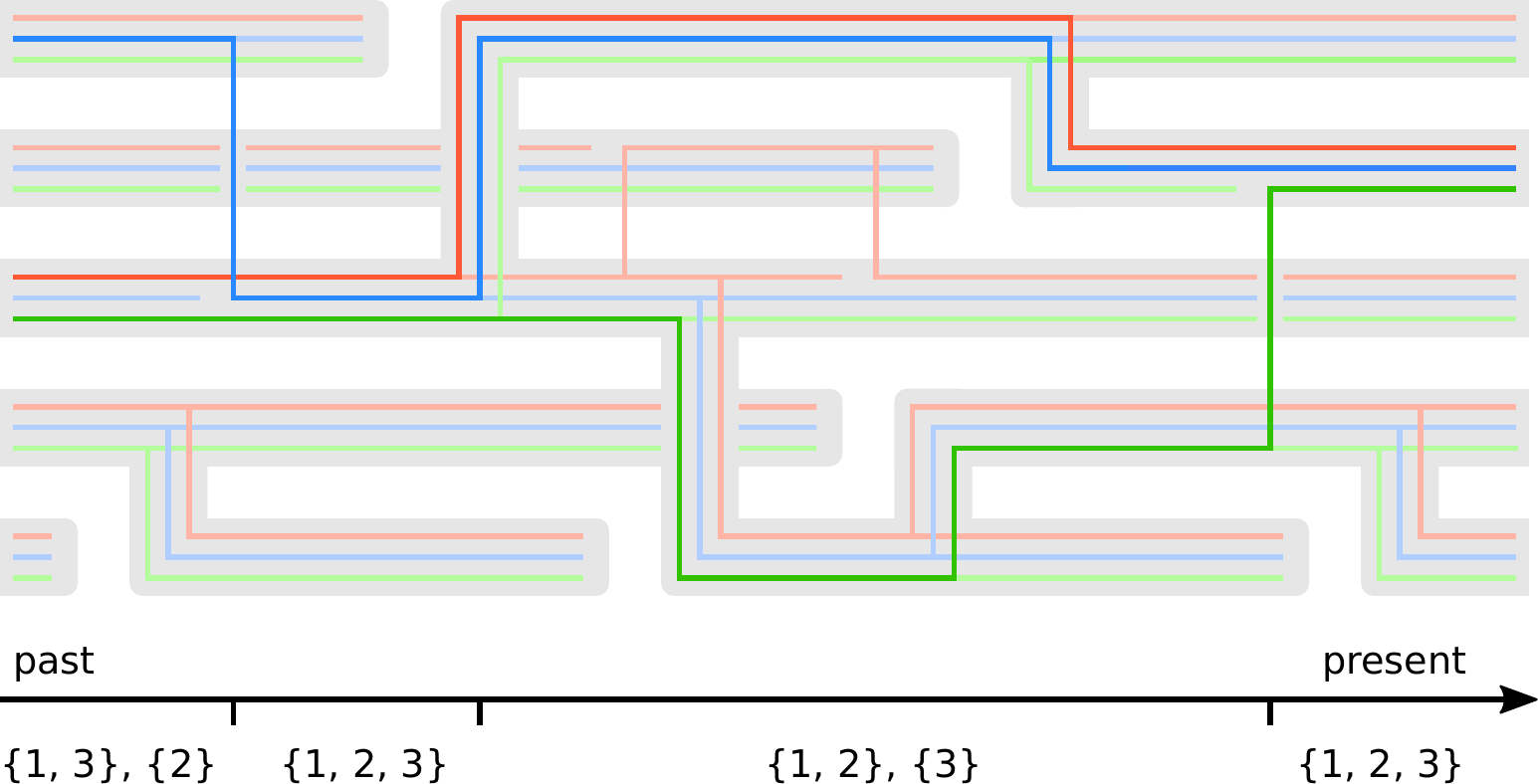}
    \caption{Illustration of the model with $N = 5$ species, represented
        by grey tubes, and $n = 3$ genes, represented by the colored
        lines inside the tubes. A species can split into two, simultaneously replicating its
        genome (speciation). A gene can replicate and move from one species to another and then replace its homologous copy in the recipient species (introgression). At present time
        a randomly chosen species is sampled: the ancestral lineages of its
        genes are represented with bolder colors. The green lineage
        is first subject to an introgression event and jumps to a new
        species. It is then brought back to the same species as the other
        genes by a coalescence event. The corresponding partition-valued
        process obtained by assigning the labels $1$, $2$ and
        $3$ to the red, blue and green gene respectively is given.}
    \label{fig:model}
\end{figure}

Now consider a fixed large time $T$, and sample uniformly one species
at that time. We follow backwards in time the ancestral lineages
of its genes and the ancestral species to which those genes
belong. This induces a process valued in the partitions of $\Set{1,
\dots, n}$ by declaring that $i$ and $j$ are in the same block at time
$t$ if the ancestral lineages of genes $i$ and $j$ sampled at $T$ lied in
the same ancestral species at time $T-t$. 

At first ($t=0$), all genes belong to the same ancestral species.
Eventually this species receives a successful migrant from another
species. Backwards in time, the gene that has been transmitted
during this event is removed from its original species and placed
in the migrant's original species. Such events occur at rate $(N-1)d$ for
each gene, and the migrant species is then chosen uniformly in the population.
Once genes belong to separate species, they can be brought back to the 
same species by coalescence events. Any two species find their
common ancestor at rate one, and at such an event the genes from the 
two merging species are placed back into the same species. 

This informal description 
shows that the partition-valued process has two kinds of transitions:
each pair of blocks merges at rate one; each gene is placed
in a new uniformly chosen species at rate $(N-1)d$.
Setting the introgression rate to $d_N = d/N$ and letting $N\to\infty$,
introgression events occur at rate $d$ for each gene. At each such event the gene is
sent to a new species that does not contain any of the 
other $n-1$ ancestral gene lineages, i.e., it is placed in a singleton
block. This is the description of Kingman's coalescent with erosion, that
we now more formally introduce.

    \subsection{Kingman's coalescent with erosion}

Let $n \ge 1$, we define the $n$-Kingman coalescent with erosion as a
Markov process $(\Pi^n_t)_{t \ge 0}$ taking values in the partitions of
$[n] \defas \Set{1, \dots, n}$. Its transition rates are the following.
Started from a partition $\pi$ of $[n]$, the process jumps to any
partition $\pi'$ obtained by merging two blocks of $\pi$ at rate $1$.
Moreover, at rate $d$ for each $i \le n$, the integer $i$ is ``eroded''.
This means that if $C$ is the block of $\pi$ containing $i$, then the
process jumps to the partition $\pi'$ obtained by replacing the block
$C$ by the blocks $C \setminus \Set{i}$ and $\Set{i}$. (Obviously if 
$C = \Set{i}$, i.e., if $i$ is in a singleton block, no such transition can
occur.)

Kingman's coalescent with erosion is a special case of the more general
class of partition-valued processes called 
\emph{exchangeable fragmentation-coalescence processes}, introduced and 
studied in~\cite{berestycki_exchangeable_2004}. 
These processes are a combination of the well-studied
fragmentation processes, where blocks can only split, and coalescence
processes, where blocks are only allowed to merge. The main new feature
of combining fragmentation and coalescence is that they can balance
each other so that fragmentation-coalescence processes display non-trivial
stationary distributions. In this work we will be interested into
describing the stationary distribution associated to Kingman's 
coalescent with erosion. The following proposition, which is a
direct consequence of Theorem~8 of~\cite{berestycki_exchangeable_2004}, 
provides the existence and uniqueness of this distribution.
\begin{proposition}[\citealt{berestycki_exchangeable_2004}]
    There exists a unique process $(\Pi_t)_{t \ge 0}$ valued in the partitions
    of $\N$ such that for all $n \ge 1$, the restriction of 
    $(\Pi_t)_{t \ge 0}$ to $[n]$ is distributed as the $n$-Kingman
    coalescent with erosion. Moreover, the process 
    $(\Pi_t)_{t \ge 0}$ has a unique stationary
    distribution $\Pi$.
\end{proposition}

Kingman's coalescent with erosion is an exchangeable process
in the sense that for any finite permutation $\sigma$ of $\N$, 
\[
    (\sigma(\Pi_t))_{t \ge 0} \overset{\mathrm{(d)}}{=} (\Pi_t)_{t \ge 0}.
\]
It is then clear that the stationary distribution $\Pi$ is also
an exchangeable partition of $\N$. Exchangeable partitions of $\N$ 
are often studied through what is known as their asymptotic frequencies. Let
$\Pi = (C_1, C_2, \dots)$ be the blocks of the partition $\Pi$.
Then, Kingman's representation theorem \cite[see e.g.][]{bertoin_2006}
shows that for any $i$, the following limit exists a.s.
\[
    \lim_{n \to \infty} \frac{1}{n} \sum_{k = 1}^n \Indic{k \in C_i} =
    f_i.
\]
Let $(\beta_i)_{i \ge 1}$ be the non-increasing reordering of the
sequence $(f_i)_{i \ge 1}$. We call $(\beta_i)_{i \ge 1}$ the 
\emph{asymptotic frequencies} of $\Pi$. The sequence
$(\beta_i)_{i \ge 1}$ is such that
\[
    \beta_1 \ge \beta_2 \ge \dots \ge 0,\quad \sum_{i \ge 1} \beta_i \le 1.
\]
Such sequences are called \emph{mass-partitions}. Mass-partitions are
studied because exchangeable partitions are entirely characterized by
their asymptotic frequencies. The partition
$\Pi$ can be recovered from its asymptotic frequencies $(\beta_i)_{i \ge
1}$ through what is known as 
a \emph{paintbox procedure}. Conditionally on $(\beta_i)_{i \ge 1}$, let 
$(X_i)_{i \ge 1}$ be an independent sequence such that for $k \ge 1$, $\Prob{X_i = k} =
\beta_k$, and $\Prob{X_i = -i} = 1-\sum_{k \ge 1} \beta_k$.
Then the partition $\Pi'$ of $\N$ defined as
\[
    i \sim_{\Pi'} j \iff X_i = X_j
\]
is distributed as $\Pi$ \cite[see e.g.][]{bertoin_2006}. We see
that $i$ is in a singleton block if{f} $X_i = -i$. The set of all
singleton blocks is referred to as the \emph{dust} of $\Pi$, and the partition
$\Pi$ has dust if{f} $\sum_{i \ge 1} \beta_i < 1$.

The main characteristics of the asymptotic
frequencies of fragmentation-coalescence processes have already
been derived in~\cite{berestycki_exchangeable_2004}, 
see Theorem~8. In the case of Kingman's coalescent
with erosion, these results specialize to the following theorem.

\begin{theorem}[\citealt{berestycki_exchangeable_2004}]
    Let $(\beta_i)_{i \ge 1}$ be the asymptotic frequencies of $\Pi$, the
    stationary distribution of Kingman's coalescent with erosion.
    Then
    \[
        \sum_{i \ge 1} \beta_i = 1,\quad \forall i \ge 1,\; \beta_i > 0,\quad a.s.
    \]
    In other words, the partition $\Pi$ has infinitely many blocks, and
    no dust.
\end{theorem}

Before stating our main two results, let us motivate them.
Consider a partition $\hat{\Pi}$ obtained from a paintbox procedure on a 
random mass-partition $(\hat{\beta}_i)_{i \ge 1}$, and denote $\hat{\Pi}^n$ 
its restriction to $[n]$. There are two sources of randomness in
$\hat{\Pi}^n$. One originates from the fact that $(\hat{\beta}_i)_{i \ge 1}$
is random. Moreover, conditionally on $(\hat{\beta}_i)_{i \ge 1}$, 
$\hat{\Pi}^n$ is obtained by sampling a finite number of variables
with distribution $(\hat{\beta}_i)_{i \ge 1}$. Thus, in addition to the 
randomness of $(\hat{\beta}_i)_{i \ge 1}$, $\hat{\Pi}^n$ is subject to a 
finite sampling randomness.

Suppose that $\hat{\Pi}$ has finitely many blocks, say $N$, with asymptotic frequencies
$(\hat{\beta}_1, \dots, \hat{\beta}_N)$. When $n$ gets large, the finite sampling
effects vanish and the sizes of the blocks of $\hat{\Pi}^n$ resemble 
$(n \hat{\beta}_1, \dots, n \hat{\beta}_N)$. However, when $\hat{\Pi}$ has infinitely many
non-singleton blocks, 
there always exists a large enough $i$ such that the size of the block
with frequency $\hat{\beta}_i$ remains subject to finite sampling effects in 
$\hat{\Pi}^n$. In
this case it is not entirely straightforward to go from the asymptotic
frequencies $(\hat{\beta}_i)_{i \ge 1}$ to the size of the blocks of
$\hat{\Pi}^n$, as this involves a non-trivial sampling procedure.

In this work our task will be twofold. First, we will investigate the
size of the ``large blocks'' of $\Pi^n$ by describing the distribution
of the asymptotic frequencies $(\beta_i)_{i \ge 1}$. In order to get
an insight into the distribution of the ``small blocks'' of $\Pi^n$, we will
rather study the empirical distribution of the size of the blocks
of $\Pi^n$, for large $n$. Let us now state the corresponding results.

    \subsection{Main results}
    \label{SS:mainResults}

We show two main results in this work. One is concerned with the size
of the large blocks of Kingman's coalescent with erosion, and gives a
representation of its asymptotic frequencies in terms of an infinite
sequence of hierarchically independent diffusions. The other is concerned
with the size of the small blocks and provides the limit of the distribution of the
size of a block chosen uniformly from the stationary partition when
$n$ is large. Let us start with the former result.

\paragraph{Size of the large blocks.}
Let $(Y_i)_{i \ge 1}$ be an i.i.d.\ sequence of diffusions verifying
\[
    \forall i \ge 1,\; \diff Y_i = (1-Y_i) \diff t + \sqrt{Y_i(1-Y_i)}
    \diff W_i,
\]
started from $0$, and where $(W_i)_{i \ge 1}$ are independent Brownian
motions. It is known, see e.g.~\cite{lambert_population_2008}
Proposition~2.3.4, that each $Y_i$ is distributed as a 
Wright-Fisher diffusion conditioned on hitting $1$, and thus
we have
\[
    \forall i \ge 1,\; \lim_{t \to \infty} Y_i(t) = 1 \quad \text{a.s.}
\]
Accordingly, we set $Y_i(\infty) = 1$. We build inductively a sequence of
processes $(Z_i)_{i \ge 1}$ and time-changes $(\tau_i)_{i \ge 1}$ as
follows. Set
\[
    \forall t \ge 0,\; Z_1(t) = Y_1(t), 
    \quad \tau_1(t) = \int_0^t \frac{1}{1-Z_1(s)} \diff s.
\]
Then, suppose that $(Z_1, \dots, Z_i)$ and $(\tau_1, \dots, \tau_i)$ have
been defined, and set
\begin{gather*}
    \forall t \ge 0,\; 
    Z_{i+1}(t) = (1-Z_1(t) - \dots - Z_i(t)) Y_{i+1}(\tau_i(t)),\\
    \forall t \ge 0,\;
    \tau_{i+1}(t) = \int_0^t \frac{1}{1-Z_1(s) -\dots -Z_{i+1}(s)} \diff s.
\end{gather*}
Then we have the following representation of the asymptotic frequencies of
the stationary distribution of Kingman's coalescent with erosion.

\begin{theorem} \label{Thm:frequencies}
    Let $(Z_i)_{i \ge 1}$ be the sequence of diffusions
    defined previously. Then the non-increasing reordering of the sequence
    $(z_i)_{i \ge 1}$ defined as
    \[
        \forall i \ge 1,\; z_i = \int_0^\infty d e^{-dt} Z_i(t) \diff t,
    \]
    is distributed as the frequencies of the blocks of the stationary 
    distribution of Kingman's coalescent with erosion rate $d$.
\end{theorem}

Let us explain the intuition behind \Cref{Thm:frequencies}.
Kingman's coalescent is dual to a measure-valued process called the
Fleming-Viot process~\citep{etheridge_2000}. The Fleming-Viot process 
describes the offspring distribution of a 
population with constant size, while Kingman's coalescent gives the
genealogy of that population. By a classical duality argument, 
Kingman's coalescent at time $t$ can be obtained by sampling individuals
at time $t$ from a Fleming-Viot process and placing in the same block those
that have the same ancestor~\citep{bertoin_stochastic_2003}. The link
with \Cref{Thm:frequencies} is that the diffusions $(Z_i)_{i \ge 1}$
correspond to the sizes of the offspring of the individuals of 
a Fleming-Viot process, ordered by extinction time of their progeny, see
\Cref{S:eves}. The integral transformation is roughly due to the
fact that in Kingman's coalescent with erosion, one needs to place in the
same block the individuals that have the same ancestor at their last
erosion event, which is an exponential variable with parameter $d$. This
heuristical argument is made rigorous in \Cref{S:eves}, where
\Cref{Thm:frequencies} is proved.

\paragraph{Size of the small blocks.}
In order to capture the characteristics of the small blocks of $\Pi^n$,
we study the empirical measure of the size of the blocks of $\Pi^n$. 
Let $M^n$ be the total number of blocks of $\Pi^n$, and let $(\abs{C^n_1},
\dots, \abs{C^n_{M^n}})$ be their sizes. For each $k \ge 1$, we denote
\[
    \mu^n_k = \frac{1}{M^n} \Card(\Set{i \suchthat \abs{C^n_i} = k})
\]
the frequency of blocks of size $k$. The probability vector 
$(\mu^n_k)_{k \ge 1}$ is the empirical measure of the size of 
the blocks of $\Pi^n$. We give the following characterization of the
asymptotic law of $(\mu^n_k)_{k \ge 1}$ and $M^n$.

\begin{theorem} \label{prop:sizeErosion}
    \begin{enumerate}[label=\upshape(\roman*)]
    \item The following convergence holds in probability
    \[
        \lim_{n \to \infty} \frac{M^n}{\sqrt{n}} = \sqrt{2d}.
    \]
    \item Moreover, for each $k \ge 1$, the following convergence holds
    in probability 
    \[
        \lim_{n \to \infty} \mu^n_k = \frac{1}{2^{2k-1}} \frac{1}{k}
        \binom{2(k-1)}{k-1} = \Prob{J = k},
    \]
    where $J$ is the total progeny of a critical binary branching
    process.
    \end{enumerate}
\end{theorem}

In the previous proposition and hereafter we call critical binary
branching process the Markov process on $\N$ starting from $1$ that jumps
from $k$ to $k+1$ and from $k$ to $k-1$ at rate $k$. Its progeny is the 
sum of the initial number of particles and of the total number of birth
events, i.e., of jumps from $k$ to $k+1$, before the process is absorbed at $0$.

\begin{remark}
    It is interesting to notice that the limiting distribution of the
    vector $(\mu^n_k)_{k \ge 1}$ is determinisitc and does not depend on
    the erosion coefficient $d$.
\end{remark}

\begin{remark}
    The convergence of the vector $(\mu^n_k)_{k \ge 1}$ is equivalent to
    the convergence in probability of the empirical measure of the size of the
    blocks of $\Pi^n$ to the distribution of $J$ in the weak topology.
\end{remark}

Let us again discuss briefly the heuristic of our proof of this result.
Erosion occurs at a rate proportional to the size of the blocks, i.e.,
a block of size $k$ is eroded at rate $k$, while
coalescence events do not take the sizes of the blocks into account. As
there are only few blocks with large size in $\Pi^n$, and many small blocks, most coalescence
events occur between small blocks, while most erosion events occur within
these few large blocks. When restricting our attention to small blocks,
we can neglect erosion, and
consider that pairs of blocks coalesce at rate $1$, and that new blocks
of size $1$ appear at constant rate due to the erosion of the large
blocks. 

This heuristic led us to consider a process analogous to
Kingman's coalescent with erosion, where pairs of blocks coalesce
at rate $1$, but new singleton blocks immigrate at constant rate $d$.
We call this process Kingman's coalescent with immigration.
We consider the genealogy of a block sampled uniformly from Kingman's
coalescent with immigration. We prove that this genealogy converges,
as the immigration rate goes to infinity, to a critical binary
birth-death process. See the forthcoming
\Cref{prop:convergenceImmigration}.

\paragraph{Outline.} The remainder of the paper is organized as follows.
In \Cref{S:defImmigration} we provide two constructions of
Kingman's coalescent with immigration, as well as a coupling between
Kingman's coalescents with erosion and immigration.
\Cref{S:immigration} is then devoted to giving the genealogy
of the blocks of Kingman's coalescent with immigration. We there
prove a result analogous to \Cref{prop:sizeErosion},
see \Cref{prop:blockSizeImmigration}. \Cref{prop:sizeErosion} is proved in
\Cref{S:proofSizeErosion}, where we carry out the coupling
between Kingman's coalescents with erosion and immigration.
Finally, we prove \Cref{Thm:frequencies} in \Cref{S:eves}.

\paragraph{Possible extensions.} As we have
mentioned, Kingman's coalescent is part of the more general class of
fragmentation-coalescence processes. We now briefly discuss potential 
extensions of our results to such processes.

The main ingredient of our study of the
size of small blocks is that fragmentation is faster for larger
blocks, while coalescence occurs at the same speed regardless of the
size of the blocks. This allows us to neglect fragmentation 
and consider a purely coalescing system where new blocks immigrate due to
the fragmentation of the large blocks. First, this picture remains
valid for $\Lambda$-coalescents with erosion, but the proofs
would be more involved because computations could no longer be 
made explictly. Morever, we believe that this picture also 
remains valid for a broad class of binary fragmentation measures.
The particles that are removed from the large block would no
longer be of size one, but should not have time to split 
on the time-scale when small blocks are formed, yielding a situation
similar to the erosion case.

\Cref{Thm:frequencies} relies on a construction of the
stationary distribution of Kingman's coalescent with erosion
from a Fleming-Viot process that can be directly generalized to $\Lambda$-coalescents
with erosion (and even to $\Xi$-coalescents with erosion) by 
using the corresponding $\Lambda$-Fleming-Viot process. However, the
explicit expression of the size of the blocks in terms of hierarchically
independent diffusions cannot be achieved in general. Nevertheless see
the end of \Cref{S:eves} for a discussion of a possible extension of
\Cref{Thm:frequencies} to Beta-coalescents with erosion.

Overall, the techniques and ideas we use in this work are not entirely 
specific to Kingman's coalescent with erosion. Nevertheless, in this 
case, the proofs are greatly simplified because all calculations can
be made explicitly. This reason led us to restrict our
attention to Kingman's coalescent with erosion in this work, and
to leave possible extensions for future work.

\section{Kingman's coalescent with immigration} 
\label{S:defImmigration}

In this section we construct Kingman's coalescent with
immigration as a partition-valued process such that 
pairs of blocks coalesce at rate $1$ and new blocks immigrate
at rate $d$. Then, we give an alternative construction of 
Kingman's coalescent with erosion from the flow of 
bridges of~\cite{bertoin_stochastic_2003}. 
Finally, the coupling between Kingman's coalescents
with erosion and with immigration is carried out in
\Cref{SS:couplingErosion}.

    \subsection{Definition}
    \label{SS:lookdown}

Consider a Poisson point process on $\R$ with intensity $d \diff t$,
and let $(T_i)_{i \in \Z}$ be its atoms labeled in increasing order
such that $T_{0} < 0 < T_1$. The sequence $(T_i)_{i \in \Z}$
corresponds to the immigration times of new particles in
the system.

Fix $N \in \Z$, we will first define Kingman's coalescent
with immigration for the particles that have a label larger
that $N$, and then extend it to all particles by consistency.
We do that in the following way. Initially, set
\[
    \forall t < T_N,\; \bar{\Pi}^N_t = \emptyset.
\]
We now extend $\bar{\Pi}^N_t$ to all real times by induction.
Suppose that $\bar{\Pi}^N_t$ has been defined on 
$\OpenInterval{-\infty, T_k}$, for $k \ge N$. We first set 
\[
    \bar\Pi^N_{T_k} = \bar\Pi^N_{T_k-} \cup \Set{k}   
\]
to represent the immigration of the new particle with label $k$.
We now let each pair of blocks of $\bar{\Pi}^N_t$ coalesce
at rate one for $T_k \le t < T_{k+1}$. One can achieve this
by considering, conditional on $\Set{\bar{\Pi}^N_{T_k} = \bar{\pi}_k}$,
an independent version $(\Pi^k_t)_{t \ge 0}$ of Kingman's coalescent
started from $\bar{\pi}_k$, and setting
\[
    \forall t < T_{k+1} - T_k,\; \bar{\Pi}^N_{T_k + t} = \Pi^k_t.
\]

We say that the process $(\bar{\Pi}^N_t)_{t \in \R}$ is 
the \emph{$N$-Kingman coalescent with immigration rate $d$}. The
following proposition shows that we can extend consistently
the $N$-Kingman's coalescent with immigration to a process taking
its values in the partitions of $\Z$, and
that it is a Markov process whose transitions coincide 
with our intuitive description of Kingman's coalescent
with immigration.

\begin{proposition}
    \begin{enumerate}[label=\upshape(\roman*)]
        \item There exists a unique process $(\bar{\Pi}_t)_{t \in \R}$,
        called Kingman's coalescent with immigration rate $d$,
        such that for all $N \in \Z$, its restriction 
        to $\Set{i \in \Z \suchthat i \ge N}$ is distributed as the
        $N$-Kingman coalescent with immigration.
        \item With probability one, $\bar{\Pi}_t$ has finitely many blocks
        for all $t \in \R$.
        \item The process $(\bar{\Pi}_t)_{t \in \R}$ is Markovian.
        Conditional on $\Set{\bar{\Pi}_t = \bar{\pi}}$, where $\bar{\pi}$
        is a partition of $\Set{i \in \Z \suchthat i \le n}$, each
        pair of blocks coalesce at rate $1$, and at rate $d$ the
        process goes to the partition $\bar\pi \cup \Set{n+1}$, i.e.,
        a new particle immigrates.
    \end{enumerate}
\end{proposition}

\begin{proof}
    (i) Let $(\bar{\Pi}^N_t)_{t \in \R}$ be a $N$-Kingman's coalescent
    with immigration. It is sufficient to
    show that the restriction $(\bar{\Pi}^{N+1}_t)_{t \in \R}$ of
    $(\bar{\Pi}^N_t)_{t \in \R}$ to $\Set{i \in \Z \suchthat i \ge N+1}$
    is distributed as a $N+1$-Kingman's coalescent with immigration, and
    the result will follow from Kolmogorov's extension theorem.
    Obviously, the immigration times of $(\bar{\Pi}^{N+1}_t)_{t \in \R}$
    have the desired distribution. The result is now a simple consequence
    of the sampling consistency of Kingman's coalescent.

    \medskip

    (ii) Let us now prove the second point. Kingman's coalescent has the
    property of coming down from infinity~\citep{kingman_coalescent_1982}.
    This means that even if Kingman's coalescent is started from a
    partition with an infinite number of blocks, then for all positive
    times it will have only finitely many blocks. Thus, as the number of
    immigrated particles is locally finite, Kingman's coalescent with
    immigration only has a finite number of blocks for all times a.s.

    \medskip

    (iii) That each $(\bar{\Pi}^N_t)_{t \in \R}$ is a Markov process is a
    direct consequence of the Markov property of Kingman's coalescent,
    and of the fact that the immigration times are distributed according
    to an independent Poisson point process with intensity $d$. This
    readily implies the Markov property of $(\bar{\Pi}_t)_{t \in \R}$.
\end{proof}

An interesting consequence of the last result is that the process
counting the number of blocks of Kingman's coalescent with immigration
is a Markov birth-death process. More precisely, for $t \in \R$, let
$M_t$ be the number of blocks of the partition $\bar{\Pi}_t$. Then
$(M_t)_{t \in \R}$ is a stationary birth-death process.
\begin{corollary} \label{cor:immigrationBlockMarkov}
    The process $(M_t)_{t \in \R}$ counting the number of blocks
    of Kingman's coalescent with immigration rate $d$ is a stationary
    Markov process. Conditional on $\Set{M_t = k}$, it jumps to
    \begin{itemize}
        \item $k+1$ at rate $d$.
        \item $k-1$ at rate $k(k-1)/2$.
    \end{itemize}
\end{corollary}

\subsection{Preliminaries on flows of bridges} \label{SS:bridges}

The previous construction of the Kingman coalescent with immigration
is based on Kolmogorov's extension theorem.
The aim of the next two sections is to give an alternative
construction of Kingman's coalescent with immigration based on 
the flow of bridges of~\cite{bertoin_stochastic_2003}. 
This construction will only be needed in \Cref{S:proofSizeErosion}
for the proof of \Cref{Thm:frequencies}. In this section we recall
the material on flows of bridges that will be needed.

\paragraph{Bridges.}
We call bridge~\citep{bertoin_stochastic_2003} any random function of the form
\[
    \forall u \in [0, 1],\; B(u) = (1 - \sum_{i \ge 1} \beta_i) u + \sum_{i \ge 1} \beta_i \Indic{u \ge V_i},
\]
for some random mass-partition $(\beta_i)_{i \ge 1}$ and an independent
i.i.d.\ sequence of uniform $[0, 1]$ variables $(V_i)_{i \ge 1}$.
For a bridge $B$, we define its inverse $B^{-1}$ as
\[
    \forall u \in \COInterval{0, 1},\; B^{-1}(u) = \inf \Set{t \in [0, 1] \suchthat
    B(t) > u}, \quad B^{-1}(1) = 1.
\]
Let $(U_i)_{i \ge 1}$ be a sequence of i.i.d.\ uniform variables.
An exchangeable partition $\hat{\Pi}$ of $\N$ can be obtained from $B$
and $(U_i)_{i \ge 1}$ by setting
\[
    i \sim_{\hat{\Pi}} j \iff B^{-1}(U_i) = B^{-1}(U_j).
\]
Let $(C_1, C_2, \dots)$ be the blocks of $\hat{\Pi}$ labeled in decreasing
order of their least elements, i.e., such that 
\[
    i \le j \iff \min(C_i) \le \min(C_j).
\]
To each block $C_i$ is associated a unique random variable $V'_i$ defined 
as 
\[
    \forall j \in C_i,\; V'_i = B^{-1}(U_j).
\]
If $\hat\Pi$ has finitely many blocks, say $M$, for $i > M$ we set $V'_i = \tilde{V}'_i$
where $(\tilde{V}'_i)_{i \ge 1}$ is an independent sequence of i.i.d.\ uniform
random variables. The sequence $(V'_i)_{i \ge 1}$ will be referred to as 
the sequence of ancestors of the blocks of $\hat{\Pi}$.
The key results on bridges from~\cite{bertoin_stochastic_2003}
is their Lemma~2 that we state here for later use.
\begin{lemma}[\citealt{bertoin_stochastic_2003}] \label{lem:bridges}
    Consider a bridge $B$, and let $\hat{\Pi}$ and $(V'_i)$ be
    respectively the partition and sequence of ancestors obtained from
    $B$ as above. Then $(V'_i)_{i \ge 1}$ is independent of $\hat{\Pi}$, 
    and $(V'_i)_{i \ge 1}$ is a sequence of i.i.d.\ uniform variables.
\end{lemma}

\paragraph{The standard flow of bridges.} A flow of bridges is defined as
follows.
\begin{definition}
    A flow of bridges is a family of bridges $(B_{s,t})_{s \le t}$ such
    that:
    \begin{enumerate}[label=\upshape(\roman*)]
        \item For any $s \le u \le t$, we have $B_{s,u} \circ B_{u,t} = B_{s,t}$.
        \item For $p \ge 1$ and $t_1 \le \dots \le t_p$, the bridges 
        $B_{t_1, t_2}, \dots, B_{t_{p-1}, t_p}$ are independent, and $B_{t_1, t_2}$
        is distributed as $B_{0, t_2-t_1}$.
        \item The limit $B_{0,t} \to \Id$ as $t \downarrow 0$ holds in
        probability in the Skorohod space.
    \end{enumerate}
\end{definition}
A flow of bridges encodes the dynamics of a population represented by the
interval $[0, 1]$. Let $t \in \R$ and $x < y$. If the interval $[x, y]$
is interpreted as a subfamily of the population at time $t$, then 
its progeny at time $s \le t$ is represented by the interval
$[B_{s,t}(x-), B_{s,t}(y)]$. (Notice that time is going backward:
if $t$ is the present, then $s \le t$ represents the future of the 
population.)

\medskip

By the independence and stationarity of the increments of the flow, the
distribution of a flow of bridges is entirely characterized by the distribution of $B_{0,t}$,
for $t \ge 0$. We will be particularly interested into the so-called
\emph{standard flow of bridges}, that can be described as follows.
Let $t \ge 0$ and consider the bridge
\[
    \forall u \in [0, 1],\; B_{0,t}(u) = \sum_{i = 1}^{N_t} \beta_i \Indic{V_i \le u},
\]
where
\begin{enumerate}[label=\upshape(\roman*)]
    \item The process $(N_t)_{t \ge 0}$ is distributed as a pure-death process started
    at $\infty$, and going from $k$ to $k-1$ at rate $k(k-1)/2$.
    \item Conditionally on $N_t$, $(\beta_1, \dots, \beta_{N_t})$ has
    a Dirichlet distribution with parameter $(1, \dots, 1)$.
    \item The variables $(V_i)_{i \ge 1}$ is an independent i.i.d.\ sequence of uniform
    variables.
\end{enumerate}
Then we know~\citep{bertoin_stochastic_2003} that there exists a flow
of bridges $(B_{s,t})_{s \le t}$ such that $B_{0,t}$ is distributed as
above. It is called the standard flow of bridges.

\medskip

Our interest in the standard flow of bridges is that is represents the
dynamics of a population whose genealogy is given by Kingman's
coalescent.  Let $(U_i)_{i \ge 1}$ be a sequence of i.i.d.\ uniform
variables, and let $\hat\Pi_t$ be the partition obtained from the bridge
$B_{0,t}$ and the sequence $(U_i)_{i \ge 1}$. We stress that the
\emph{same} sequence is used for all $t$. Then the process $(\hat\Pi_t)_{t \ge 0}$ 
is distributed as Kingman's coalescent started from the partition
of $\N$ into singletons~\citep{bertoin_stochastic_2003}.

\paragraph{The Fleming-Viot process.} One of the main advantages of flows
of bridges is that they couple a backward process, giving the genealogy 
of the population, and a forward process, giving the size of the progeny
of the individuals in the population. This forward process is often
encoded as a measure-valued process known as a Fleming-Viot process.

Let $(B_{s,t})_{s \le t}$ be a standard flow of bridges. For each $t \ge
0$, $B_{-t,0}$ is the distribution function of some random measure
$\rho_t$ on $[0, 1]$. The measure-valued process $(\rho_t)_{t \ge 0}$ is
called a Fleming-Viot process~\citep{etheridge_2000}. A well-known fact that we will use
is that the dynamics of the mass of a fixed interval 
is a Wright-Fisher diffusion. More precisely, let $x \in [0, 1]$ and 
$X_t = \rho_t([0, x])$. Then the process $(X_t)_{t \ge 0}$ is a Wright-Fisher
diffusion started from $x$, i.e., it is distributed as the unique solution to 
\[
    \diff X = \sqrt{X(1-X)} \diff W,\quad X_0 = x, 
\] 
where $W$ is a standard Brownian motion.

    \subsection{A flow of bridges construction of Kingman's coalescent
    with immigration}
    \label{SS:coupling}

Let $(B_{s,t})_{s \le t}$ be a standard flow of bridges. We now construct
a version of Kingman's coalescent with immigration from $(B_{s,t})_{s \le t}$. 
Consider a Poisson point process on $\R \times [0, 1]$ with
intensity $d \diff t \otimes \diff x$, and let $(T_i, U_i)_{i \in \Z}$ be
its atoms, labeled in increasing order of their first coordinate such
that $T_0 < 0 < T_1$.
Similarly to \Cref{SS:lookdown}, the times $(T_i)_{i \in \Z}$
correspond to immigration times of new particles. Here the sequence
$(U_i)_{i \in \Z}$ represents the location in the population of these
immigrated particles.

For each $t \in \R$, we define a partition $\bar{\Pi}_t$ of 
$\Set{i \in \Z \suchthat T_i \le t}$ by setting
\[
    i \sim_{\bar{\Pi}_t} j \iff B^{-1}_{T_i, t}(U_i) = B^{-1}_{T_j, t}(U_j).
\]
The following proposition shows that $(\bar{\Pi}_t)_{t \in \R}$
is distributed as Kingman's coalescent with immigration.
\begin{proposition} \label{prop:bridgeConstruction}
    The process $(\bar{\Pi}_t)_{t \in \R}$ defined from the flow 
    of bridges is a version of Kingman's coalescent with immigration rate $d$.
\end{proposition}
\begin{proof}
    The proof almost identical to the proof of Corollary~1
    of \cite{bertoin_stochastic_2003}. The main difference is that 
    here the flow of bridges is sampled at various times $(T_i)_{i \in \Z}$
    while for the classical Kingman coalescent, the flow of bridges is
    only sampled at an initial time.
    
    We work conditionally on $(T_i)_{i \in \Z}$ and consider these
    times as fixed. It is sufficient to show that for all $N \in \Z$,
    between immigration times the blocks of $(\bar{\Pi}^N_t)_{t \in \R}$
    coalesce according to independent versions of Kingman's coalescent.

    Let $t \in \R$, and let $(C_1, \dots, C_{M_t})$
    be the blocks of $\bar{\Pi}^N_t$, where $M_t$ is the number of blocks, 
    and where the blocks are labeled such that
    \[
        i \le j \iff \min(C_i) \le \min(C_j).
    \]
    Similarly to \Cref{SS:bridges}, we can define the sequence of ancestors
    of $\bar{\Pi}^N_t$ by setting
    \[
        \forall j \in C_i,\; V'_i = B^{-1}_{T_j, t}(U_j),
    \]
    and supplementing it with an independent sequence of i.i.d.\ uniform
    variables $(\tilde{V}'_i)_{i \ge 1}$, i.e., defining $\forall i > M_t$, 
    $V'_i = \tilde{V}'_i$.

    \medskip

    Let us show by induction that for all $k \ge N$, 
    \begin{enumerate}[label=\upshape(\roman*)]        
        \item The ancestors $(V^{(k)}_i)_{i \ge 1}$ of $\bar{\Pi}^N_{T_k}$ are i.i.d.\ with
        uniform distribution.
        \item The sequence $(V^{(k)}_i)_{i \ge 1}$ is independent of
        $(\bar{\Pi}^N_t)_{t \le T_k}$.
        \item $(\bar{\Pi}^N_t)_{t \le T_k}$ is a version of the $N$-Kingman
        coalescent with immigration.
    \end{enumerate}

    Fix $T_k \le t_1 < \dots < t_{p+1} \le T_{k+1}$. By induction on $p$ we can
    suppose that the sequence of ancestors of $\bar{\Pi}^N_{t_p}$,
    denoted by $(V^{(t_p)}_i)_{i \ge 1}$, is independent of $\big( (\bar{\Pi}^N_t)_{t \le T_k},
    \bar{\Pi}^N_{t_1}, \dots, \bar{\Pi}^N_{t_p} \big)$. Then (i) and (ii)
    are proved if we can show that the sequence of ancestors of 
    $\bar{\Pi}^N_{t_{p+1}}$ is independent of $\big( (\bar{\Pi}^N_t)_{t \le T_k},
    \bar{\Pi}^N_{t_1}, \dots, \bar{\Pi}^N_{t_{p+1}} \big)$.

    Let us now call $\Pi^*$ the partition obtained from the bridge 
    $B_{t_p, t_{p+1}}$ and the sequence $(V^{(t_p)}_i)_{i \ge 1}$, i.e.,
    \[
        i \sim_{\Pi^*} j \iff 
        B^{-1}_{t_p, t_{p+1}}(V^{(t_p)}_i) = B^{-1}_{t_p, t_{p+1}}(V^{(t_p)}_j)     ,
    \] 
    and let $(V^*_i)_{i \ge 1}$ be the sequence of ancestors of $\Pi^*$,
    i.e.,
    \[
        \forall j \in C^*_i,\; V^*_i = B^{-1}_{t_p, t_{p+1}}(V^{(t_p)}_j),
    \]
    where $(C^*_1, C^*_2, \dots)$ denote the blocks of $\Pi^*$ labeled in 
    increasing order of their minimal elements as above.    
    Using the fact that for $u \le s \le t$, $B^{-1}_{u, t} = B^{-1}_{s, t} \circ B^{-1}_{u, s}$, 
    we get that for all $N \le i, j \le k$, 
    \begin{align}
        i \sim_{\bar{\Pi}_{t_{p+1}}} j &\iff 
        B^{-1}_{t_p, t_{p+1}}(B^{-1}_{T_i, t_p}(U_i)) = 
        B^{-1}_{t_p, t_{p+1}}(B^{-1}_{T_j, t_p}(U_j)) \nonumber\\
        &\iff 
        B^{-1}_{t_p, t_{p+1}}(V^{(t_p)}_{b(i)}) = 
        B^{-1}_{t_p, t_{p+1}}(V^{(t_p)}_{b(j)})\nonumber\\
        &\iff b(i) \sim_{\Pi^*} b(j) \label{eq:bridges}
    \end{align}
    where $b(i)$ denotes the label of the block of $\bar{\Pi}^N_{t_p}$
    to which $i$ belongs. 

    By independence of the increments of the flow of bridges, 
    the bridge $B_{t_p, t_{p+1}}$ is independent of the collection
    of variables $\big((\bar{\Pi}^N_t)_{t \le T_k}, \bar{\Pi}^N_{t_1},
    \dots, \bar{\Pi}^N_{t_p}, (V^{(t_p)}_i)_{i \ge 1}\big)$. 
    Thus, $(B_{t_p, t_{p+1}}, (V^{(t_p)}_i)_{i \ge 1})$
    are independent of $\big((\bar{\Pi}^N_t)_{t \le T_k}, \bar{\Pi}^N_{t_1},
    \dots, \bar{\Pi}^N_{t_p}\big)$, and hence $(\Pi^*, (V^*_i)_{i \ge 1})$ 
    are independent of $\big((\bar{\Pi}^N_t)_{t \le T_k}, \bar{\Pi}^N_{t_1},
    \dots, \bar{\Pi}^N_{t_p}\big)$.
    Using \Cref{lem:bridges}, we get that $\Pi^*$ is independent of $(V^*_i)_{i \ge 1}$.
    This shows that $(V^*_i)_{i \ge 1}$ is independent  
    $\big((\bar{\Pi}^N_t)_{t \le T_k}, \bar{\Pi}^N_{t_1}, \dots, \bar{\Pi}^N_{t_p}, \Pi^*\big)$. 
    Using~\eqref{eq:bridges}, we see that $\bar{\Pi}^N_{t_{p+1}}$ can be
    recovered from $\bar{\Pi}^N_{t_p}$ and $\Pi^*$. Thus, the variables
    $\big((\bar{\Pi}^N_t)_{t \le T_k}, \bar{\Pi}^N_{t_1}, \dots, \bar{\Pi}^N_{t_{p+1}}\big)$ 
    are independent of $(V^*_i)_{i \ge 1}$.
    
    In order to end the proof of the claim we need to distinguish two
    cases. First, suppose that $t_{p+1} < T_{k+1}$. Then, due to our
    labeling convention, we have that $(V^*_i)_{i \ge 1} =
    (V^{(t_{p+1})}_i)_{i \ge 1}$ (up to the auxiliary variables that play
    no role). Conversely, if $t_{p+1} = T_{k+1}$, then one of the 
    variables $(V^*_i)_{i \ge 1}$ has to be replaced by the ancestor 
    $U_{k+1}$ of the block $\Set{k+1}$. More precisely, if 
    $\bar{\Pi}^N_{T_{k+1}}$ has $M_{k+1}$ blocks, again by labeling
    convention, the block $\Set{k+1}$ has label $M_{k+1}$. Thus,
    $(V^{(t_{p+1})}_i)_{i \ge 1}$ is recovered by setting 
    $V^{(t_{p+1})}_i = V^*_i$ for $i \ne M_{k+1}$, and 
    $V^{(t_{p+1})}_i = U_{k+1}$ for $i = M_{k+1}$. It is straightforward
    to see that as $U_{k+1}$ is independent of all other variables,
    the sequence $(V^{(t_{p+1})}_i)_{i \ge 1}$ remains independent of
    $\big((\bar{\Pi}^N_t)_{t \le T_k}, \bar{\Pi}^N_{t_1}, \dots, \bar{\Pi}^N_{t_{p+1}}\big)$
    and thus that points (i) and (ii) of the claim hold.

    \medskip

    For $k \ge N$ and $t < T_{k+1} - T_k$ consider the partition
    $\Pi^k_t$ of $\Set{i \in \Z \suchthat N \le i \le k}$ defined as
    \[
        i \sim_{\Pi^k_t} j \iff B^{-1}_{T_k, T_k + t}(V^{(k)}_{b(i)}) 
        = B^{-1}_{T_k, T_k + t}(V^{(k)}_{b(j)}) 
    \]
    where $b(i)$ is the label of the block of $\bar{\Pi}^N_{T_k}$ to
    which $i$ belongs. As the sequence $(V^{(k)}_i)_{i \ge 1}$
    is i.i.d.\ uniform, the process $(\Pi^k_t)_{t < T_{k+1} - T_k}$
    is a version of Kingman's coalescent started from
    $\bar{\Pi}^N_{T_k}$. The that fact these coalescents are independent
    is a consequence of the previous induction. This proves (iii),
    and ends the proof of the result.
\end{proof}
    
    \subsection{Coupling erosion and immigration}
    \label{SS:couplingErosion}

We now explain the coupling between Kingman's coalescents with erosion
and with immigration. Let $n \ge 1$, consider a Poisson point process $P^n$
on $\R$ with intensity $nd \diff t$ and let $(T_i)_{i \in \Z}$ be
its atoms ordered increasingly such that $T_{0} < 0 < T_1$. To each atom of the process we
attach a uniform mark in $[n]$. We denote by $\ell_i$ the mark attached to $T_i$, so
that $(\ell_i)_{i \in \Z}$ is a sequence of i.i.d.\ uniform variables on
$[n]$.

Consider $t \in \R$. For each $k \in [n]$, let $\phi_t(k)$ be the label of the
last atom of $P^n$ with mark $k$, i.e., $\phi_t(k) \in \Z$ is the unique 
$i$ such that $\ell_i = k$ and there is no atom $T$ of $P^n$ with $T_i <
T \le t$ carrying mark $k$.
Let $(\bar{\Pi}_t)_{t \in \R}$ be Kingman's coalescent with immigration
rate $nd$ built from the Poisson process $(T_i)_{i \in \Z}$ as in
\Cref{SS:lookdown}. We define a partition $\Pi^n_t$ of
$[n]$ by setting
\[
    i \sim_{\Pi^n_t} j \iff \phi_t(i) \sim_{\bar{\Pi}_t} \phi_t(j).
\]
In words, $i$ and $j$ belong to the same block of $\Pi^n_t$ if{f} the last
particles of $(\bar{\Pi}_t)_{t \in \R}$ with marks $i$ and $j$ have coalesced before time $t$.
The key point of this construction is that $(\Pi^n_t)_{t \in \R}$ is 
distributed as Kingman's coalescent with erosion.
\begin{proposition} \label{prop:erosionImmigration}
    The process $(\Pi^n_t)_{t \in \R}$ is a stationary version of the 
    $n$-Kingman coalescent with erosion rate $d$.
\end{proposition}
\begin{proof}
    Let $k \in [n]$. By thinning, the set of atoms of $P^n$ with mark $k$
    is a Poisson process on $\R$ with intensity $d \diff t$, and these
    processes are independent. Thus new
    atoms of $P^n$ with mark $k$ arrive at rate $d$. Let us consider
    what happens at such an arrival time. Suppose that 
    $\ell_i = k$. Then, by definition, we have $\phi_{T_i}(k) = i$,
    as the atom $T_i$ has mark $k$. Moreover, the particle $i$
    is a singleton of the partition $\bar{\Pi}_{T_i}$ (it is the 
    particle that has newly immigrated). Thus at time $T_i$, the
    integer $k$ is removed from its block and placed in a singleton
    block. This is the description of an erosion event, which occur
    at rate $d$.

    Let us now describe the dynamics between immigration times. The
    atoms of $P^n$ that are the last atoms with their marks form a 
    subset of the atoms $P^n$. By sampling consistency of Kingman's 
    coalescent, the restriction of the process $(\bar{\Pi}_t)_{t \in \R}$
    to these atoms is also distributed as Kingman's coalescent.
    Thus any two pairs of blocks of such atoms with a last mark coalesce
    at rate one, and so does the blocks of $(\Pi_t)_{t \in \R}$.

    The fact that $(\Pi_t)_{t \in \R}$ is stationary follows from
    the stationarity of the Poisson point process.
\end{proof}

Combined with the construction of Kingman's coalescent with immigration
from the standard flow of bridges, this coupling gives an interesting
construction of the stationary distribution of Kingman's coalescent with
erosion.
\begin{corollary} \label{cor:erosionLookdown}
    Let $(B_{s,t})_{s \le t}$ be a standard flow of bridges, 
    $(T_i)_{i \ge 1}$ and $(U_i)_{i \ge 1}$ be independent sequences
    of i.i.d.\ exponential variables with parameter $d$, and of uniform
    variables respectively. Then the partition $\Pi$ defined by
    \[
        i \sim_{\Pi} j \iff B^{-1}_{-T_i, 0}(U_i) = B^{-1}_{-T_j, 0}(U_j)
    \]
    has the stationary distribution of Kingman's coalescent with erosion
    rate $d$.
\end{corollary}
\begin{proof}
    Consider a Poisson process $P^n$ on $\R \times [0, 1]$ with intensity
    $n d \diff t \otimes \diff x$, and attach to each atom of $P^n$ a uniform mark
    on $[n]$. If $(T_i, U_i)$ denotes the last atom of $P^n$ with mark
    $i$ before $t=0$, then $T_i$ is exponentially distributed with parameter $d$, 
    $U_i$ is uniform on $[0, 1]$, and all these variables are
    independent. A combination of
    \Cref{prop:erosionImmigration} and
    \Cref{prop:bridgeConstruction} now proves the result.
\end{proof}

\begin{remark}
    The construction of Kingman's coalescent with immigration from
    \Cref{SS:lookdown} and the construction with the flow of bridges of
    \Cref{SS:coupling} only rely on the sampling consistency of 
    Kingman's coalescent. These constructions could be extended directly
    to a case where the coalescence events occur according to a 
    $\Lambda$-coalescent~\citep{pitman_coalescents_1999,
    sagitov_general_1999}.
    In particular, the construction of the stationary distribution
    of Kingman's coalescent with erosion of
    \Cref{cor:erosionLookdown} extends directly to
    $\Lambda$-coalescents with erosion if one replaces the standard
    flow of bridges by the corresponding $\Lambda$-flow of bridges.
\end{remark}

\section{Size of the blocks of Kingman's coalescent with immigration}
\label{S:immigration}

In this section we study Kingman's coalescent with immigration.
The main result we will show is the following.

\begin{proposition} \label{prop:blockSizeImmigration}
    Let $n \ge 1$ and consider $(\bar{\Pi}^n_t)_{t \in \R}$ a version of Kingman's
    coalescent with immigration rate $nd$. Let $(\abs{\bar{C}^n_1},
    \dots, \abs{\bar{C}^n_p})$ be the size of $p$ blocks chosen uniformly from
    $\bar{\Pi}^n_0$, then
    \[
        (\abs{\bar{C}^n_1}, \dots, \abs{\bar{C}^n_p}) \Longrightarrow (J_1, \dots, J_p)
    \]
    where $(J_1, \dots, J_p)$ are i.i.d.\ variables distributed as
    the total progeny of a critical binary branching process.
\end{proposition}

We prove this result by choosing $k$ blocks uniformly from $\bar{\Pi}^n_0$,
and counting backwards in time the number of blocks that are
ancestors of these blocks, i.e., that will further coalesce to
form these blocks. We show that this process converges, under
appropriate scaling, to $k$ independent critical binary branching
processes, yielding the result.

We first give a precise definition of the ancestral process counting the
number of blocks in \Cref{SS:ancestral}, along with its
basic properties. The convergence is then carried out in
\Cref{SS:ancestralConvergence}.

    \subsection{The ancestral process} \label{SS:ancestral}

Let $(\bar{\Pi}_t)_{t \in \R}$ be a version of Kingman's coalescent with
immigration rate $d$. The process $(\bar{\Pi}_t)_{t \in \R}$ is naturally
endowed with a notion of ancestry between its blocks. 
For $t \in \R$, let $M_t$ be the number of blocks of $\bar{\Pi}_t$.
Let $(\bar{C}_1, \dots, \bar{C}_{M_t})$ be an enumeration of the blocks
of $\bar{\Pi}_t$. We say that this enumeration is exchangeable
if conditional on $\Set{M_t = k}$, for any permutation $\sigma$ of
$[k]$, 
\[
    (\bar{C}_1, \dots, \bar{C}_{k}) \overset{\text{(d)}}{=} 
    (\bar{C}_{\sigma(1)}, \dots, \bar{C}_{\sigma(k)}).
\]
We can always consider an exchangeable enumeration of
the blocks of $\bar{\Pi}_t$ by changing the labels of any enumeration
according to an independent uniform permutation.

For $s \le t$, consider $\bar{\Pi}_t = (\bar{C}_1, \dots, \bar{C}_{M_t})$ and 
$\bar{\Pi}_s = (\bar{C}'_1, \dots, \bar{C}'_{M_s})$ an enumeration of the 
blocks of $\bar{\Pi}_t$ and $\bar{\Pi}_s$ respectively. In Kingman's
coalescent with immigration, a block present at time $s$ can only
coalesce with other blocks. Thus, for any block $\bar{C}'_i$, there is a
unique block $\bar{C}_j$ of $\bar{\Pi}_t$ such that $\bar{C}'_i \subseteq
\bar{C}_j$. We say that $\bar{C}'_i$ is an ancestor of $\bar{C}_j$.

\begin{definition}
    Let $(\bar{\Pi}_t)_{t \ge 0}$ be Kingman's coalescent with
    immigration, and let $(\bar{C}_1, \dots, \bar{C}_{M_0})$
    be the blocks of $\bar{\Pi}_0$ enumerated in an exchangeable
    order. For each $t \ge 0$ and $i \le M_0$, we define $\Ai_t(i)$
    to be the number of blocks of $\bar{\Pi}_{-t}$ that are ancestors of
    $\bar{C}_{i}$. We set $\Ai_t(i) = 0$ for $i > M_0$. Then defining
    $\Ai_t \defas (\Ai_t(1), \Ai_t(2), \dots)$, the process $(\Ai_t)_{t \ge 0}$ 
    is called the ancestral process associated to $(\bar{\Pi}_t)_{t \in \R}$.
\end{definition}

The definition of the ancestral process is illustrated in \Cref{fig:ancestral}.
The process $(\Ai_t)_{t \ge 0}$ can be seen as a particle system where at time $0$,
there are $M_0$ particles with distinct types, and $(\Ai_t(i))_{t \ge 0}$ 
records the number of particles with type $i$. As we have reversed
time, each coalescence event now corresponds to the birth of a new
particle, and each immigration event to the death of a particle.

\begin{figure}
    \center
    \includegraphics{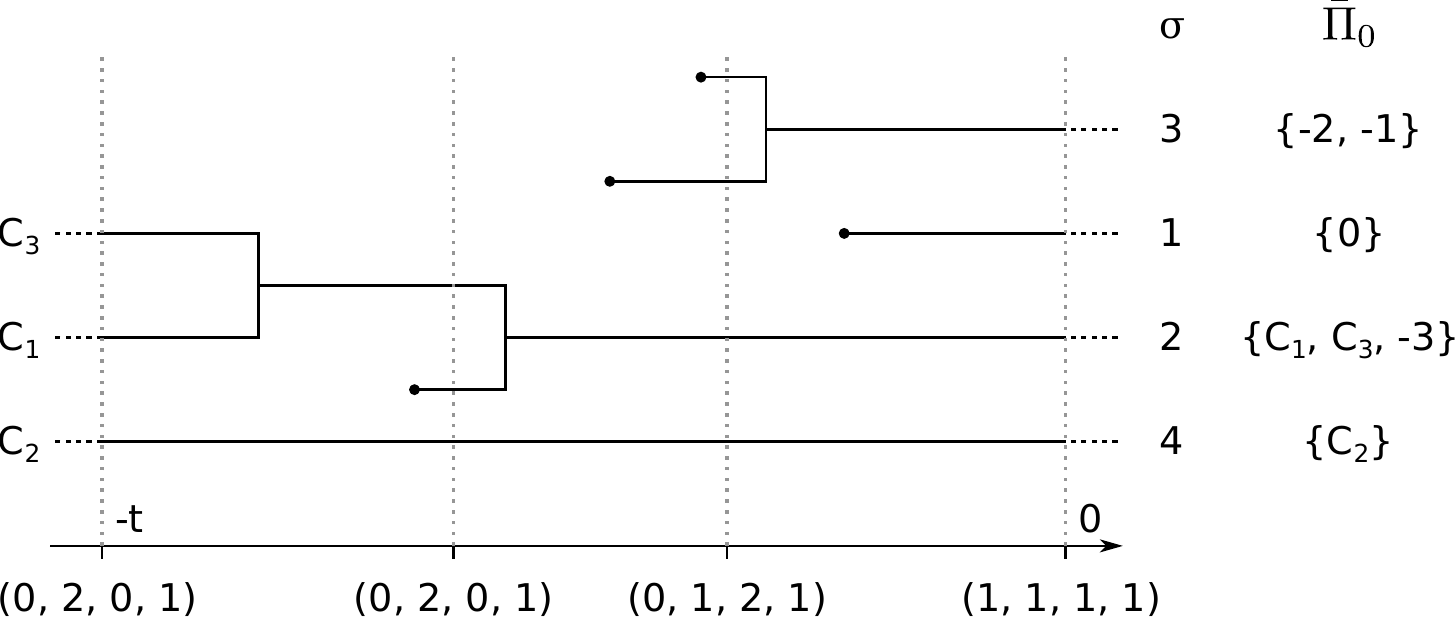}
    \caption{In this example, we have $\bar{\Pi}_{-t} = (C_1, C_2, C_3)$.
    Each black circle represents an immigration event, and the lines
    merge at the coalescence time of the blocks to which they correspond.
    At $t = 0$ the blocks of $\bar{\Pi}_0$ are labeled according to the 
    permutation $\sigma$, and the value of $(\Ai_t)_{t \ge 0}$ is given
    below for some times.}
    \label{fig:ancestral}
\end{figure}

\bigskip

Recall that $(M_t)_{t \in \R}$ stands for the number of blocks
of $(\bar\Pi_t)_{t \in \R}$ forward in time. For each $t \in \R$, we
define $N_t \defas M_{-t}$, the number of blocks of $(\bar{\Pi}_t)_{t \in
\R}$ backwards in time. The process $(N_t)_{t \ge 0}$ also gives the number of 
particles of the ancestral process $(\Ai_t)_{t \ge 0}$, that is we
have 
\[
    \forall t \ge 0,\; N_t = \sum_{i \ge 1} \Ai_t(i).
\] 
The following proposition shows that the ancestral process is Markovian.
This is a key feature that makes Kingman's coalescent with immigration
easier to study than Kingman's coalescent with erosion.

\begin{proposition} \label{prop:ancestralMarkov}
    Let $(\Ai_t)_{t \ge 0}$ be the ancestral process associated
    to Kingman's coalescent with immigration rate $d$, and
    let $(N_t)_{t \ge 0}$ be the number of particles of $(\Ai_t)_{t \ge
    0}$. Then $(\Ai_t)_{t \ge 0}$ is a Markov process with initial
    condition
    \[
        \forall i \le N_0,\; \Ai_0(i) = 1, 
        \quad \forall i > N_0,\; \Ai_0(i) = 0.
    \]
    Moreover, conditionally on $\Ai_t$:
    \begin{itemize}
        \item each particle gives birth to a new particle of its type
        at rate $d/N_t$.
        \item each particle dies at rate $(N_t-1)/2$.
    \end{itemize}
\end{proposition}

The proof of \Cref{prop:ancestralMarkov} can be found in
\Cref{A:counting}, we only sketch it here. The process $(M_t)_{t \in \R}$
is a stationary birth-death process, with rates given in
\Cref{cor:immigrationBlockMarkov}. A simple calculation
shows that it is actually a reversible process, i.e., with our notation,
that $(N_t)_{t \ge 0}$
is distributed as $(M_t)_{t \ge 0}$. When $(N_t)_{t \ge 0}$ jumps 
from $k$ to $k+1$, a particle has given birth to two particles. By
exchangeability of our system, the particle that gives birth is
chosen uniformly, i.e., each particle gives birth at the same rate 
$d/k$. Similarly, when $(N_t)_{t \ge 0}$ jumps from $k$ to $k-1$
a particle chosen uniformly from the population dies. Thus each
particle dies at rate $k(k-1)/(2k) = (k-1)/2$.

Making the above argument rigorous involves counting the number of
trajectories of $(\bar{\Pi}_t)_{t \in \R}$ yielding a given trajectory
of $(\Ai_t)_{t \ge 0}$. We postpone it until \Cref{A:counting}.

\bigskip

In order to prove \Cref{prop:blockSizeImmigration}, we need
to keep track of the number of ancestors of $k$ blocks chosen
uniformly from $\bar{\Pi}_0$. As we have chosen a uniform labeling of the 
blocks of $\bar{\Pi}_0$, this amounts to considering the process
$(\Ai_t(1), \dots, \Ai_t(k);\, t \ge 0)$.
\Cref{prop:ancestralMarkov} directly gives us the distribution
of this process.

\begin{corollary} \label{cor:sampleMarkov}
    The process $(\Ai_t(1), \dots, \Ai_t(p), N_t;\, t \ge 0)$ is 
    a Markov process such that conditional on 
    $\Set{\Ai_t(1) = a_1, \dots, \Ai_t(p) = a_p, N_t = k}$, the process
    jumps to:
    \begin{itemize}
        \item $(a_1, \dots, a_i + 1, \dots, a_p, k+1)$ at rate $\frac{d}{k} a_i$.
        \item $(a_1, \dots, a_i - 1, \dots, a_p, k-1)$ at rate $\frac{k-1}{2}a_i$.
        \item $(a_1, \dots, a_p, k+1)$ at rate $\frac{d}{k}(k-a_1-\dots-a_p)$.
        \item $(a_1, \dots, a_p, k-1)$ at rate $\frac{k-1}{2}(k-a_1-\dots-a_p)$.
    \end{itemize}
\end{corollary}

\begin{proof}
    We see from the expression of the transition rates of $(\Ai_t)_{t \ge
    0}$ that the rate at which each particle splits or dies only depends
    on the rest of the population through the total population 
    size $N_t$. This is enough to prove the result. 
\end{proof}

    \subsection{Convergence}
    \label{SS:ancestralConvergence}

We now prove that the process $(\Ai_t(1), \dots, \Ai_t(p);\, t \ge 0)$
converges to independent critical binary birth-death processes when time is
rescaled by a factor $1/\sqrt{n}$. We start with the following lemma.
\begin{lemma} \label{lem:tightness}
    Let $M^n$ have the stationary distribution of $(M^n_t)_{t \ge 0}$, 
    the number of blocks of Kingman's coalescent with immigration rate
    $dn$. The sequence $(M^n/\sqrt{n};\, n \ge 1)$ is tight.
\end{lemma}
\begin{proof}
    Let $n \ge 1$ and consider a birth-death process $(X^n_t)_{t \ge 0}$
    such that conditional on $\Set{X^n_t = k}$, the process jumps to 
    \begin{itemize}
        \item $k+1$ at rate $dn$;
        \item $k-1$ at rate $\mu_k$,
    \end{itemize}
    where the death rate $\mu_k$ is defined as  
    \[
        \mu_k = \begin{cases}
            0 &\text{ if $k < \sqrt{2dn} + 1$,}\\
            \frac{(\sqrt{2dn}+1)\sqrt{2dn}}{2} &\text{ else.}
        \end{cases}
    \]
    The process $(X^n_t - \lfloor\sqrt{2dn}+1\rfloor;\, t \ge 0)$ is
    distributed as a simple random walk, reflected at $0$. Thus it admits a
    geometric stationary distribution with parameter $\gamma_n$ given
    by
    \[
        \gamma_n =
        \frac{2dn}{(\sqrt{2dn}+1)\sqrt{2dn}}
        = \frac{1}{1 + \sqrt{\frac{1}{2dn}}}.
    \]
    This shows that the process $(X^n_t)_{t \ge 0}$ also
    admits a stationary distribution. If $X^n$ has the 
    stationary distribution of $(X^n_t)_{t \ge 0}$, then
    $X^n$ is distributed as $\lfloor\sqrt{2dn}\rfloor + 1 + Y^n$, where $Y^n$
    has a geometric distribution with parameter $\gamma_n$.

    Hence, for $K$ and $n$ large enough, we have
    \begin{align*}
        \Prob{X^n \le K \sqrt{n}} &\le \Prob{Y^n \le K \sqrt{n} -
        \sqrt{2dn}}\\
        &= 1-\gamma_n^{(K-\sqrt{2d})\sqrt{n}}\\
        &= 1-\exp(-\frac{K-\sqrt{2d}}{\sqrt{2d}}) + o_n(1).
    \end{align*}
    Thus the sequence $(X^n / \sqrt{n};\, n \ge 1)$ is tight.

    Recall that $(M^n_t)_{t \ge 0}$ is a birth-death process
    jumping from $k$ to $k+1$ at rate $dn$, and from $k$ to 
    $k-1$ at rate $k(k-1)/2 \ge \mu_k$. Its stationary distribution
    is thus dominated by that of $X^n$, and this proves the result.
\end{proof}

We now prove our main convergence result. The proof will use 
a result from Chapter~11 of~\cite{ethier_1986} on the a.s.\ convergence
of rescaled Markov processes. In order to stick to their notation,
we introduce
\[
    \forall t \ge 0,\; \hat{N}^n_t = N^n_{t/\sqrt{n}}, \quad
    \hatA = \Ai^n_{t/\sqrt{n}},
\]
and 
\[
    \forall x \ge 0,\; \beta_+(x) = d,\; \beta_-(x) = \frac{x^2}{2},\;
    F(x) = d - \frac{x^2}{2}
\]

\begin{proposition} \label{prop:convergenceImmigration}
    Let $(\Ai^n_t)_{t \ge 0}$ be the ancestral process of
    Kingman's coalescent with immigration rate $dn$.
    Then 
    \[
        \big( \hatA(1), \dots, \hatA(p),
        \frac{\hat{N}^n_t}{\sqrt{n}};\, t \ge 0 \big)
        \Longrightarrow 
        \big( X_1(t), \dots, X_p(t), \sqrt{2d};\, t \ge 0 \big),
    \]
    in the sense of convergence in distribution in the Skorohod space,
    and where the processes $(X_1, \dots, X_p)$ are
    i.i.d.\ critical binary birth-death processes, with per-capita birth
    and death rate $\sqrt{d/2}$.
\end{proposition}

\begin{proof}
    We start by showing that the process
    $(\hat{N}^n_t/\sqrt{n};\, t \ge 0)$
    converges to the constant process with value $\sqrt{2d}$.
    The process
    $(\hat{N}^n_t)_{t \ge 0}$ is a Markov process jumping from
    \begin{itemize}
        \item $k$ to $k+1$ at rate $d\sqrt{n} = \sqrt{n}
        \beta_+(\frac{k}{\sqrt{n}})$.
        \item $k$ to $k-1$ at rate 
            $\frac{k(k-1)}{2\sqrt{n}} = \sqrt{n}
            \beta_-(\frac{k}{\sqrt{n}}) -
        \frac{1}{2\sqrt{n}}$.
    \end{itemize}
    Thus, the process $(\hat{N}^n_t)_{t \ge 0}$ is of the same form
    as the processes considered in Theorem~2.1 of Chapter~11
    of~\cite{ethier_1986}, except that the scaling is $\sqrt{n}$ and not
    $n$.

    Let us consider a stationary version of the process 
    $(\hat{N}^n_t)_{t \ge 0}$. \Cref{lem:tightness} shows that the
    sequence
    $(\hat{N}^n_0/\sqrt{n};\, n \ge 1)$ is tight. We can thus
    find an increasing sequence of indices $(n_k)_{k \ge 1}$ such that the
    subsequence $(\hat{N}^{n_k}_0/\sqrt{n_k};\, k \ge 1)$ converges in distribution
    to a limiting variables $N$. Using Skorohod's representation
    theorem \cite[see e.g.][]{billingsley_1999}, we can assume that the
    convergence holds a.s.

    Applying Theorem~2.1 of Chapter~11 of~\cite{ethier_1986} shows that 
    the sequence of processes $(\hat{N}^{n_k}_t/\sqrt{n_k};\, t \ge 0, k \ge 1)$ converges
    a.s.\ uniformly on compact sets to the solution of 
    \begin{align} \label{eq:diffEquation}
        \dot{x} = F(x) = d - \frac{x^2}{2},
    \end{align}
    started from the random variables $N$. (The original theorem
    is given for a different scaling, but the proof is easily adapted
    to ours.) As each process $(\hat{N}^{n_k}_t)_{t \ge 0}$ is 
    stationary, the limiting process is a stationary 
    solution to~\eqref{eq:diffEquation}, i.e., is the constant
    process with value $\sqrt{2d}$. This shows that each
    converging subsequence of $(\hat{N}^n_t / \sqrt{n};\, t \ge 0, n \ge 1)$
     converges to the same constant process, and thus that the
    entire sequence converges. 

    \medskip 

    Let us now prove the convergence of the ancestral processes. Consider
    independent Poisson processes $(P^-_i(t))_{t \ge 0}$, $(P^+_i(t))_{t
    \ge 0}$ for $i \le p$, and $(P^-_N(t))_{t \ge 0}$, 
    $(P^+_N(t))_{t \ge 0}$. Using e.g.\ Theorem~4.1 from Chapter~6
    of~\cite{ethier_1986}, there exists a unique strong solution to the
    following equation
    \begin{gather*}
        \forall t \ge 0, \forall i \le p,\;
        X^n_i(t) = P^+_i\Big(\int_0^t \frac{d \sqrt{n}X^n_i(s)}{Y^n(s)} \diff s\Big) 
        - P^-_i\Big(\int_0^t \frac{X^n_i(s)(Y^n(s)-1)}{2 \sqrt{n}} \diff s\Big),\\
        Y^n(t) = P^+_N\Big(\int_0^t d\sqrt{n} (1-\tfrac{\sum_i X_i^n(s)}{Y^n(s)}) \diff s \Big)
        - P^-_N\Big(\int_0^t \tfrac{Y^n(s)(Y^n(s)-1)}{2\sqrt{n}} (1-\tfrac{\sum_i X^n_i(s)}{Y^n(s)}) \diff s \Big)
        + \sum_{i=1}^p X^n_i(t).
    \end{gather*}
    Moreover, this solution $(X^n_1, \dots, X^n_p, Y^n)$ is distributed
    as $(\hatA(1), \dots, \hatA(p), \hat{N}^n_t;\, t \ge 0)$.

    As $Y^n/\sqrt{n}$ converges in probability to the constant process with value
    $\sqrt{2d}$, we can find a subsequence such that 
    \[
        \lim_{n \to \infty} \frac{d\sqrt{n}}{Y^n(t)} =
        \sqrt{\frac{d}{2}},\quad
        \lim_{n \to \infty} \frac{(Y^n(t)-1)}{2\sqrt{n}} = 
        \sqrt{\frac{d}{2}} \quad \text{a.s.}
    \]
    holds uniformly in $t$ on compact sets. This is sufficient to show
    that for each $i \le p$, 
    the subsequence of processes $(X^n_i(t))_{t \ge 0}$ converges a.s.\ in the 
    Skorohod space to the solution $(X_i(t))_{t \ge 0}$ of 
    \[
        \forall t \ge 0, \forall i \le p,\;
        X_i(t) = P^+_i\Big(\int_0^t \sqrt{\frac{d}{2}} X_i(s) \diff s\Big) 
        - P^-_i\Big(\int_0^t \sqrt{\frac{d}{2}} X_i(s) \diff s\Big).
    \]
    This proves that the entire sequence $(X^n_1, \dots,
    X^n_p)$ converges in probability in the Skorohod topology to the
    solution of the previous equation. Finally, noting that 
    the solutions of these equations are independent and distributed
    as critical binary branching processes with branching rate
    $\sqrt{d/2}$ ends the proof.
\end{proof}

We are now ready to prove \Cref{prop:blockSizeImmigration}.

\begin{proof}[Proof of \Cref{prop:blockSizeImmigration}]
    By construction, the size of $p$ blocks of $\bar{\Pi}^n$ chosen
    uniformly is given by the total number of particles of the processes 
    $(\hatA(1), \dots, \hatA(p);\, t \ge 0)$. Thus, in the limit, the 
    size of these blocks converges to the total size of $p$ independent
    critical binary branching processes.
\end{proof}

\section{Proof of \texorpdfstring{\Cref{prop:sizeErosion}}{Theorem~\ref{prop:sizeErosion}}}
\label{S:proofSizeErosion}

In the previous section we have derived the limiting distribution
of the sizes of blocks uniformly sampled from Kingman's coalescent with
immigration. In this section we make use of the coupling between
Kingman's coalescent with immigration and Kingman's coalescent
with erosion from \Cref{SS:coupling} to get the analogous result 
in the erosion case.

We first show the following result.

\begin{corollary}
    Let $\Pi^n$ have the stationary distribution of the
    $n$-Kingman coalescent with erosion. Let $(\abs{C^n_1}, \dots, \abs{C^n_p})$ be
    the size of $p$ blocks chosen uniformly from $\Pi^n$.
    Then
    \[
        (\abs{C^n_1}, \dots, \abs{C^n_p}) \Longrightarrow (J_1, \dots, J_p),
    \]
    where $(J_1, \dots, J_p)$ are i.i.d.\ variables distributed as
    the total progeny of a critical binary branching process.
\end{corollary}

\begin{proof}
    Recall the coupling between Kingman's coalescent with erosion
    and Kingman's coalescent with immigration. Let $(T_i)_{i \in \Z}$
    be the atoms of a Poisson point process $P^n$ with intensity $dn$,
    labeled in increasing order such that $T_0 < 0 < T_1$. Consider
    an independent i.i.d.\ sequence of marks $(\ell_i)_{i \in \Z}$
    that are uniformly distributed on $[n]$. 
    
    Let $\bar{\Pi}^n_0$ be the value at time $0$ of the version of
    Kingman's coalescent with erosion rate $nd$ built from $(T_i)_{i \in \Z}$ 
    as in \Cref{SS:lookdown}. We know from
    \Cref{prop:erosionImmigration} that we can obtain a
    version $\Pi^n$ of the stationary distribution of the $n$-Kingman
    coalescent with erosion by placing $i$ and $j$ in the same block
    of $\Pi^n$ if the last atoms of $P^n$ in $\OCInterval{-\infty, 0}$
    with mark $i$ and $j$ both belong to the same block of $\bar{\Pi}^n_0$.
    
    Now let $(\bar{C}_1^n, \dots, \bar{C}^n_p)$ be $p$ blocks chosen
    uniformly from $\bar{\Pi}_0$, and let $(\abs{\bar{C}^n_1}, \dots,
    \abs{\bar{C}^n_p})$ be their respective sizes. For $k \le p$, let 
    \[
        \abs{C^n_k} = \Card \Set{i \in \bar{C}^n_k \suchthat \text{$(T_i, \ell_i)$ is
        the last atom in $\OCInterval{-\infty, 0}$ with mark $\ell_i$}}.
    \]
    Then conditionally on $\Set{\abs{C^n_1} \ge 1, \dots, \abs{C^n_p} \ge 1}$,
    $(\abs{C^n_1}, \dots, \abs{C^n_p})$ are the sizes of $p$ blocks
    chosen uniformly from $\Pi^n$. The result is thus proved if we can
    show that
    \[
        \lim_{n \to \infty} \Prob{\abs{C^n_1} = \abs{\bar{C}^n_1}, \dots, 
        \abs{C^n_p} = \abs{\bar{C}^n_p}} = 1.
    \]

    Let us first explain intuitively why the previous claim holds. The 
    ancestors of $\bar{C}^n_1$ have all immigrated on a time-scale of
    order $1/\sqrt{n}$. On this time-scale, there are of order $\sqrt{n}$
    particles that have also immigrated. All these particles receive a
    uniform label in $[n]$. Thus the probability that an ancestor of 
    $\bar{C}^n_1$ has received the same label as one of the other
    $\sqrt{n}$ particles, i.e., that it is not the first atom with its
    mark, is of order $1/\sqrt{n}$. Let us make this argument rigorous.

    Set 
    \[
        \tau^n_1 \defas \min \Set{T_i \suchthat i \in \bar{C}^n_1}
    \]
    to be the total life-time of the ancestors of the block $\bar{C}^n_1$.
    (The variable $\tau^n_1$ gives the immigration time of the first
    particle that forms the block $\bar{C}^n_1$.)
    The total number of particles that have immigrated during the time
    interval $[\tau^n_1, 0]$ is then $P^n([\tau^n_1, 0])$. Consider the 
    event 
    \[
        E_k = \Set{\abs{\bar{C}^n_1}=k,\; \tau^n_1 \in [-\tfrac{t}{\sqrt{n}}, 0],\;
        P^n([-\tfrac{t}{\sqrt{n}}, 0]) \le (1+\epsilon)dt\sqrt{n}}.
    \]
    On this event, if $\abs{C^n_1} \ne \abs{\bar{C}^n_1}$, then one the
    $k$ ancestors of $\bar{C}^n_1$ has received the same label as one of the particle
    that has immigrated in the time interval $[\tau^n_1, 0]$, that is,
    the same label as one of the $(1+\epsilon)dt\sqrt{n}$ last atoms
    of $P^n$. As the labels are chosen uniformly, the probability that 
    the $k$ ancestors all have labels distinct from the labels of the
    $(1+\epsilon)dt\sqrt{n}$ last particles is 
    \[
        \Big(1-\frac{1}{n}\Big)\dots\Big(1-\frac{k-1}{n}\Big)
        \Big(1-\frac{k}{n}\Big)^{(1+\epsilon)dt\sqrt{n}-k}
    \]
    which goes to $1$ as $n$ goes to infinity for all fixed $k$.
    Thus
    \[
        \Prob{\abs{C^n_1} \ne \abs{\bar{C}^n_1}, E_k} \le 
        \Big(1-\frac{1}{n}\Big)\dots\Big(1-\frac{k-1}{n}\Big)
        \Big(1-\frac{k}{n}\Big)^{(1+\epsilon)dt\sqrt{n}-k},
    \]
    and 
    \begin{align*}
        \Prob{\abs{C^n_1} \ne \abs{\bar{C}^n_1}} \le \;
        &\Prob{\tau^n_1 \not\in [-\tfrac{t}{\sqrt{n}}, 0]} + \Prob{\abs{\bar{C}^n_1} \ge K}\\
        &+ \Prob{P^n([-\tfrac{t}{\sqrt{n}}, 0]) > (1+\epsilon)dt\sqrt{n}} + o_n(1).
    \end{align*}
    Now, by \Cref{prop:blockSizeImmigration}, the sequence
    $(-\sqrt{n}\tau^n_1)_{n \ge 1}$ converges in
    distribution to the total life-time of a binary critical branching
    process, and $(\abs{\bar{C}^n_1})_{n \ge 1}$ converges to the total progeny
    of this process. Thus, the first two terms in the above equation can be
    made as small as desired uniformly in $n$ by taking $t$ and $K$ large
    enough. Using Chebishev's inequality, the last term can also be made small
    by choosing a large enough $\epsilon$. This proves the result for $p = 1$
    and a simple union bound proves the result for any $p$.
\end{proof} 

\begin{remark}
    In the previous proof, on the event $\Set{\abs{\bar{C}^n_1} = \abs{C^n_1}}$,
    not only the size of the blocks of Kingman's coalescents
    with erosion and immigration coincide, but also the 
    genealogy of the blocks. Thus we have shown the slightly
    stronger result that, in the $n$-Kingman coalescent with erosion,
    the genealogy of a block chosen uniformly from the stationary
    distribution converges to that of a critical binary branching process.
\end{remark}

We can now prove \Cref{prop:sizeErosion}. Recall that
$\mu^n_k$ denotes the frequency of blocks of size $k$ of $\Pi^n$,
i.e., if the blocks of $\Pi^n$ are $(C^n_1, \dots, C^n_{M^n})$,
then
\[
    \mu^n_k = \frac{1}{M^n} \Card(\Set{i \suchthat \abs{C^n_i} = k}).
\]

\begin{proof}[Proof of \Cref{prop:sizeErosion}]
    (i) We start by proving that 
    $M^n/\sqrt{n}$ converges to $\sqrt{2d}$ in probability.
    Let us consider a version $\bar\Pi^n$ of the stationary distribution
    of Kingman's coalescent with immigration rate $nd$, coupled with a
    version $\Pi^n$ of the stationary distribution of Kingman's
    coalescent with erosion rate $d$ on $[n]$. Let $\bar{M}^n$, resp.\
    $M^n$, denote the number of blocks of $\bar{\Pi}^n$, resp.\ $\Pi^n$.
    Recall that the blocks of $\Pi^n$ are subsets of the blocks of
    $\bar{\Pi}^n$, where a particle is retained if there are no other
    particles with the same label that have immigrated after it. Let
    $\abs{\bar{C}^n}$ be the size of a block of $\bar{\Pi}^n$ chosen
    uniformly, and let $\abs{C^n}$ be the size of the corresponding block
    of $\Pi^n$. Some blocks of $\bar{\Pi}^n$ are only composed of 
    particles that are not retained to form $\Pi^n$. Such blocks have
    no corresponding blocks in $\Pi^n$, and $\bar{M}^n - M^n$ is exactly
    the number of such blocks. Thus
    \[
        \E \big[ \frac{\bar{M}^n - M^n}{\bar{M}^n} \big]
        = \Prob{\abs{C^n} = 0} \longrightarrow 0.
    \]
    This shows that $M^n / \bar{M}^n$ goes to $1$ in probability.
    \Cref{lem:tightness} further shows that $\bar{M}^n / \sqrt{n}$ goes
    to $\sqrt{2d}$ in probability, and thus that $M^n / \sqrt{n}$
    also goes to $\sqrt{2d}$ in probability.

    \medskip 

    (ii) We prove the second point using the method of moments.
    Let $(\abs{C^n_1}, \dots, \abs{C^n_p})$ be the sizes of $k$ uniformly
    sampled blocks of $\Pi^n$. Then, as the number of blocks
    $M^n$ goes to infinity, we have that 
    \[
        \lim_{n \to \infty} \E[(\mu^n_k)^p] = \lim_{n \to \infty}
        \Prob{\abs{C^n_1} = \dots = \abs{C^n_p} = k} = \Prob{J = k}^p,
    \]
    where $J$ is the total progeny of a binary critical branching
    process. The convergence of the moments readily implies convergence
    in distribution as the limit is a Dirac mass.
\end{proof}

\section{Asymptotic frequencies of Kingman's coalescent with erosion}
\label{S:eves}

In this section we prove \Cref{Thm:frequencies}, which gives 
a representation of the asymptotic frequencies in terms of hierarchically
independent diffusions. First, we use the
flow of bridges construction of Kingman's coalescent with erosion from
\Cref{cor:erosionLookdown} to give a correspondence between
the frequencies of the blocks and the size of the families of a Fleming-Viot
process.

    \subsection{Eves of a Fleming-Viot process}

Let $(\rho_t)_{t \ge 0}$ be a Fleming-Viot process. For each 
individual $x \in [0, 1]$, denote
\[
    \zeta(x) = \inf \Set{t \ge 0 \suchthat \rho_t(\Set{x}) = 0}
\]
the extinction time of the offspring of $x$. It is clear 
that the set
\[
    \Set{x \in [0, 1] \suchthat \zeta(x) > 0}
    = \Set{x \in [0, 1] \suchthat \text{$\rho_t(\Set{x}) > 0$ for some $t
    \ge 0$}}
\]
is countable. The elements of this set can actually be enumerated in 
decreasing order of their extinction time, that is, they
can be written $(\e_i)_{i \ge 0}$ with
\[
    \zeta(\e_1) > \zeta(\e_2) > \dots
\]
This fact can be found e.g.\ in~\cite{labbe_flows_2014}, Theorem~1.6.
The sequence $(\e_i)_{i \ge 0}$ is called the sequence of \emph{Eves}
of $(\rho_t)_{t \ge 0}$, and was introduced
in~\cite{bertoin_stochastic_2003} and~\cite{labbe_flows_2014},
see also~\cite{duquesne_eve_2014} for a similar notion for
Continuous-State Branching Processes. The following result
shows that the frequencies of the blocks of the stationary distribution
of Kingman's coalescent with erosion can be recovered from the 
size of the offspring of the Eves.

\begin{lemma}
    Let $(\e_i)_{i \ge 1}$ be the Eves of a Fleming-Viot process
    $(\rho_t)_{t \ge 0}$. Then the non-increasing reordering
    of the sequence $(z_i)_{i \ge 1}$ defined as 
    \[
        \forall i \ge 1,\; z_i = \int_0^\infty de^{-dt}
        \rho_t(\Set{\e_i}) \diff t
    \]
    is distributed as the frequencies of the blocks of the
    stationary distribution of Kingman's coalescent with erosion
    rate $d$.
\end{lemma}
\begin{proof}
    Consider a flow of bridges $(B_{s,t})_{s \le t}$, and let $(T_i)_{i \ge 1}$,
    $(U_i)_{i \ge 1}$ be two independent i.i.d.\ sequences of exponential
    variables with parameter $d$, and uniform variables respectively.
    Again, let $\Pi$ be the partition of $\N$ defined as
    \[
        i \sim_{\Pi} j \iff B^{-1}_{-T_i, 0}(U_i) = B^{-1}_{-T_j, 0}(U_j),
    \]
    which has the stationary distribution of Kingman's coalescent with
    erosion. We denote $\Pi = (C_1, C_2, \dots)$ the blocks of $\Pi$,
    ordered in increasing order of their least elements, i.e., such that
    \[
        i \le j \iff \min(C_i) \le \min(C_j).
    \]
    Then let us call 
    \[
        A_i = B^{-1}_{-T_j, 0}(U_j),\; \forall j \in C_i,
    \]
    the ancestor of the block $C_i$. 

    As the flow of bridges $(B_{s,t})_{s \le t}$ is independent
    of the sequences $(U_i)_{i \ge 1}$ and $(T_i)_{i \ge 1}$,
    the sequence $(B^{-1}_{-T_i, 0}(U_i))_{i \ge 1}$ is exchangeable.
    Thus, the law of large numbers shows that for any $i \ge 1$, 
    \[
        \frac{1}{n} \Card(C_i \cap [n]) = \frac{1}{n} \sum_{j = 1}^n
        \Indic{B^{-1}_{-T_j, 0}(U_j) = A_i} 
        \underset{n \to \infty}{\longrightarrow} \int_0^\infty d e^{-dt}
        \rho_t(\Set{A_i}) \diff t \quad \text{a.s.}
    \]
    Thus the result is proved if we can show that a.s.
    \[
        \Set{\e_i \suchthat i \ge 1} = \Set{A_i \suchthat i \ge 1}.
    \]
    Clearly we have $\zeta(A_i) > 0$, as otherwise the frequency of the
    block $C_i$ would be zero. Moreover, conditionally on the flow of bridges,
    there exists a.s.\ some $j \ge 1$ such that 
    \[
        (U_j, T_j) \in \Set{(x,t) \suchthat B^{-1}_{-t, 0}(x) = \e_i}
    \]
    as by definition of $\e_i$ this set has positive Lebesgue measure. Thus, 
    a.s.\ $\e_i$ is the ancestor of some block of $\Pi$, and the result
    is proved.
\end{proof}

In order to prove \Cref{Thm:frequencies}, it remains to show that
the sequence of processes $\big(\rho_t(\Set{\e_1}),\allowbreak
\rho_t(\Set{\e_2}), \dots ;\, t \ge 0\big)$ has the same distribution as
the sequence of hierarchically independent diffusions introduced in
\Cref{SS:mainResults}. In the following section we characterize
this distribution, and complete the proof in the last section.

    \subsection{Wright-Fisher diffusion conditioned on its extinction
    order}
    \label{SS:WFconditioned}

Consider a $n$-dimensional Wright-Fisher diffusion $(X_1, \dots, X_n)$.
That is, $(X_1, \dots, X_n)$ is distributed as the unique solution to
\[
    \forall i \ge 1,\; \diff X_i = \sum_{\substack{j = 1\\j \ne i}}^n \sqrt{X_i X_j} \diff
    W_{i,j},
\]
where $(W_{i,j})_{i < j}$ are independent Brownian motions, and $W_{j,i}
= -W_{i,j}$, and started from an initial condition $(x_1, \dots, x_n) \in
\OpenInterval{0, 1}^n$ verifying $x_1 + \dots + x_n = 1$.
The Wright-Fisher diffusion describes the dynamics of a population
with constant size, where individuals can be of $n$ different types; 
$X_i$ denotes the frequency of type $i$ individuals in the population.
Each process $X_i$ is eventually absorbed at $0$ or $1$. We say that
the family $X_i$ reaches fixation if it gets absorbed at $1$, and that it 
becomes extinct otherwise. Let 
\[
    \zeta_i = \inf \Set{t \ge 0 \suchthat X_i = 0}
\]
denote its absorption time at $0$. 

In this section, we study the distribution of $(X_1, \dots, X_n)$
conditionally on the event $\Set{\zeta_n < \dots < \zeta_1}$.
First, notice that as $X_1 + \dots + X_n = 1$, there is exactly
one family that reaches fixation. Thus, on the event $\Set{\zeta_n <
\dots < \zeta_1}$, we have $\zeta_1 = \infty$ and $X_1$ reaches fixation;
$X_2$ is the last family to go extinct, and $X_n$ is the first
family to go extinct.
We now express the distribution of the conditioned Wright-Fisher
diffusion in terms of the diffusions introduced in
\Cref{SS:mainResults}.

\bigskip 

We will work inductively, by first conditioning the process $(X_1, \dots,
X_n)$ on $\zeta_1$ being the largest extinction time, then on $\zeta_2$
being the second largest and so on and so forth. The key point is that after conditioning
on the fixation of $X_1$, the remainder of the population, $(X_2, \dots,
X_n)$, is distributed as a rescaled, time-changed, unconditioned
$(n-1)$-dimensional Wright-Fisher diffusion, independent of $X_1$.

Let us be more specific and let $Y_1$ be the solution of
\begin{align}
    \diff Y_1 = (1-Y_1) \diff t + \sqrt{Y_1(1-Y_1)} \diff W_1, \label{eq:hTransform}
\end{align}
for some Brownian motion $W_1$. Notice that $Y_1$ is distributed
as a usual $1$-dimensional Wright-Fisher diffusion, conditioned on
fixation. Consider the fixation time of $Y_1$ which is defined as
\[
    S_1 = \inf \Set{t \ge 0 \suchthat Y_1(t) = 1}.
\]
We further define a random time-change $\tau_1$ as 
\[
    \forall t < S_1,\; \tau_1(t) = \int_0^t
    \frac{1}{1-Y_1(s)} \diff s, \quad \forall t \ge S_1,\;
    \tau_1(t) = \infty.
\]
We start by proving the following result.

\begin{lemma} \label{lem:firstConditioning}
    Let $Y_1$ and $\tau_1$ be as above and consider an independent
    $(n-1)$-dimensional Wright-Fisher diffusion $(X_2, \dots, X_n)$. Then,
    the process $(Z_1, \dots, Z_n)$ defined as  
    \begin{gather*}
        Z_1 = Y_1, \quad \forall i > 1, \forall t \ge 0,\; Z_i(t) = (1-Z_1(t))
        X_i(\tau_1(t)),
    \end{gather*}
    is distributed as a $n$-dimensional Wright-Fisher diffusion
    conditioned on $\Set{\zeta_1 = \infty}$.
\end{lemma}

\begin{remark}
    The time $\tau_1(t)$ is infinite with positive probability. However,
    each of the processes $(X_2, \dots, X_n)$ has an a.s.\ limit
    as $t$ goes to infinity. On the event $\Set{\tau_1(t) = \infty}$,
    we take $X_i(\tau_1(t))$ to be this limit, so that the process
    $(Z_1, \dots, Z_n)$ is now well-defined.
\end{remark}

Before proving \Cref{lem:firstConditioning}, we need the following 
fact that we prove for the sake of completeness.

\begin{lemma} \label{lem:brownianIntegral}
    Let $(W_t)_{t \ge 0}$ be a Brownian motion on $\R$ started at $1$,
    and let $T_0$ be the first time it hits $0$. Then for $\alpha \in \R$,
    a.s.
    \[
        \int_0^{T_0} W_s^\alpha \diff s = 
        \begin{cases}
            \infty &\text{ if $\alpha \le -2$}\\
            y_\alpha < \infty &\text{ if $\alpha > -2$.}
        \end{cases}
    \]
\end{lemma}

\begin{proof}
    Let us define 
    \[
        \forall t \ge 0,\; \xi_t = \tilde{W}_t - \frac{t}{2}, \quad 
        \tau(t) = \inf \Big\{s \ge 0 : \int_0^s \exp(2 \xi_u) \diff u > t\Big\},
    \]
    for a Brownian motion $(\tilde{W}_t)_{t \ge 0}$ with the convention that
    $\inf \emptyset = \infty$ and $\xi_\infty = -\infty$. The Lamperti 
    representation of positive self-similar processes~\citep{lamperti_semistable_1972} shows that $W_t$ stopped at
    $T_0$ satisfies the equality in distribution
    \[
        (W_{t \wedge T_0})_{t \ge 0}
        \overset{\mathrm{(d)}}{=} (\exp(\xi_{\tau(t)}))_{t \ge 0}.
    \]
    Thus
    \[
        \int_0^{t \wedge T_0} W_s^\alpha \diff s \overset{\mathrm{(d)}}{=}
        \int_0^t \exp(\alpha \xi_{\tau(s)}) \diff s
        = \int_0^{\tau(t)} \exp((2+\alpha) \xi_{s}) \diff s,
    \]
    and 
    \[
        \int_0^{T_0} W_s^\alpha \diff s \overset{\mathrm{(d)}}{=}
        \int_0^\infty \exp((2+\alpha) \xi_{s}) \diff s,
    \]
    which yields the result.
\end{proof}

\begin{proof}[Proof of \Cref{lem:firstConditioning}]
    Consider a $n$-dimensional Wright-Fisher diffusion $(X_1, \dots,
    X_n)$. A calculation of 
    Doob's $h$-transform using the harmonic function
    \[
        h(x_1, \dots, x_n) = \Prob{\lim_{t \to \infty}
        X_1(t) = 1 \given X_1(0) = x_1, \dots, X_n(0) = x_n} = x_1
    \]
    shows that the process $(X_1, \dots, X_n)$ conditioned on
    $\Set{\lim_{t \to \infty} X_1(t) = 1} = \Set{\zeta_1 = \infty}$ is
    distributed as the unique solution to the equation
    \begin{gather*}
        \diff X_1 = (1-X_1) \diff t + \sum_{\substack{j=2}}^n \sqrt{X_1 X_j} \diff W_{1,j},\\
        \forall i \ge 2,\; \diff X_i = -X_i \diff t + \sum_{\substack{j =1\\j \ne i}}^n
        \sqrt{X_i X_j} \diff W_{i,j},
    \end{gather*}
    where $(W_{i,j})_{i < j}$ are independent Brownian motions, and
    $W_{i,j} = -W_{j,i}$. We will prove that the process $(Z_1, \dots,
    Z_n)$ solves this equation.

    \medskip

    Now consider a $(n-1)$-dimensional Wright-Fisher diffusion $(X'_2, \dots, X'_n)$
    independent of $Y_1$ which solves
    \begin{gather*}
        \forall i \ge 2,\; \diff X'_i = \sum_{\substack{j = 2\\j \ne
        i}}^n \sqrt{X'_i X'_j} \diff W'_{i,j}.
    \end{gather*}
    We start by giving the equation solved by the process $(Y_1, X'_2 \circ
    \tau_1, \dots, X'_n \circ \tau_1)$. Notice that here, only a 
    subset of the processes are time-changed, and that $\tau_1$
    explodes in finite time. For these two reasons, let us realize
    the time-change carefully.
    
    We transform $\tau_1$ into a family of finite stopping times.
    Our first task is to prove that $\tau_1$ goes continuously to
    infinity, we do this using the speed and scale measures of the diffusion $Y_1$, see
    e.g.~\cite{etheridge_2011}. If we define $D = 1 / Y_1$, then
    \[
        \diff D = \sqrt{D-1} D \diff W_1, \quad 
        \forall t \ge 0,\; [D, D]_t = \int_0^t (D(s)-1) D(s)^2 \diff s.
    \]
    Thus we can write that
    \[
        \int_0^{S_1} \frac{1}{1-Y_1(s)} \diff s = \int_0^{S_1}
        \frac{D(s)}{D(s)-1} \diff s \overset{\mathrm{(d)}}{=}
        \int_0^{S_1} \frac{W_1([D,D]_s)}{W_1([D,D]_s)-1} \diff s =
        \int_0^{T_1} \frac{1}{(W_1(s)-1)^2W_1(s)} \diff s
    \]
    where $W_1$ is a Brownian motion started at $1/Y_1(0)$, and
    $T_1$ is the first time when $W_1$ hits $1$.
    We now know from \Cref{lem:brownianIntegral} that this integral
    is a.s.\ infinite, and thus that $\tau_1$ goes continuously to
    infinity, and does not ``jump to infinity''.

    Further consider the times
    \[
        \forall i \ge 2,\; S_i = \inf \Set{t \ge 0 \suchthat X'_i(t) = 1}, 
        \quad S = \min(S_2, \dots, S_n).
    \]
    At time $S$, one of the families has reached fixation, and thus 
    for $t \ge S$ we have $X'_i(t) = X'_i(S)$. Therefore, for all
    $t \ge 0$, we have $X'_i(\tau_1(t)) = X'_i(\tau_1(t) \wedge S)$, where
    the stopping time $\tau_1(t) \wedge S$ is now a.s.\ finite, and
    $t \mapsto \tau_1(t) \wedge S$ is continuous. (The continuity
    requires that $\tau_1$ does not jump to infinity.)
    Thus, by making a time-change in the following integrals,
    see e.g.~\cite{kallenberg_foundations_2002}, Theorem~17.24, we obtain
    \begin{align*}
        \forall t \ge 0,\; X'_i(\tau_1(t))
        &= X'_i(\tau_1(t) \wedge S) \\
        &= \sum_{\substack{j = 2\\ j \ne i}}^n \int_0^{\tau_1(t) \wedge S}
        \sqrt{X'_i(s)X'_j(s)} \diff W_{i,j} \\
        &= \sum_{\substack{j = 2\\ j \ne i}}^n \int_0^{t}
        \sqrt{X'_i(\tau_1(s) \wedge S)X'_j(\tau_1(s) \wedge S)} \diff
        W_{i,j}(\tau_1(s) \wedge S) \\
        &= \sum_{\substack{j = 2\\ j \ne i}}^n \int_0^{t}
        \sqrt{\frac{X'_i(\tau_1(s))X'_j(\tau_1(s))}{1-Y_1(s)}} \diff \tilde{W}_{i,j}
    \end{align*}
    where 
    \[
        \forall t \ge 0,\; \tilde{W}_{i,j}(t) = \int_0^t \sqrt{1-Y_1(s)}
        \diff W_{i,j}(\tau_1(s) \wedge S).
    \]

    A direct computation of the quadratic variations gives 
    \[
        \forall i,j, t \ge 0,\; [\tilde{W}_{i,j}, \tilde{W}_{i,j}]_t = t \wedge S,
    \]
    and the crossed variations are null. Thus
    a multidimensional version of Dubins-Schwarz theorem, see e.g.\
    Theorem~18.4 in~\cite{kallenberg_foundations_2002},
    shows that we can find independent Brownian motions 
    $(\hat{W}_{i, j})_{i < j}$ such that $\tilde{W}_{i,j}(t) =
    \hat{W}_{i,j}(t \wedge S)$. This proves that the time-changed processes
    solve
    \[
        \forall t \ge 0,\; X'_i(\tau_1(t)) = \sum_{\substack{j = 2\\j \ne i}}^n
        \int_0^t \sqrt{\frac{X'_i(\tau_1(s)) X'_j(\tau_1(s))}{1-Y_1(s)}} \diff \hat{W}_{i,j}.
    \]

    A final application of It\^o's formula shows that the process 
    $(Z_1, \dots, Z_n)$ as defined above solves the same equation
    as $(X_1, \dots, X_n)$ conditioned on $\Set{\zeta_1 = \infty}$.
    This proves the result.
\end{proof}

We can now proceed inductively. Let us set up the notation for the
proof. Consider i.i.d.\ processes $(Y_1, \dots, Y_{n-1})$ such that 
\[
    \forall i \ge 1,\; \diff Y_i = 
    (1-Y_i) \diff t + \sqrt{Y_i(1-Y_i)} \diff W_i
\]
where $(W_1, \dots, W_{n-1})$ are independent Brownian motions. We
set $\tilde{Z}_1 = Y_1$, and 
\[
    \forall t \ge 0,\; \tilde{\tau}_1(t) = \int_0^t \frac{1}{1-\tilde{Z}_1(s)} \diff s.
\]
We then define recursively, for $i < n-1$,
\begin{gather*}
    \forall t \ge 0,\; \tilde{Z}_{i+1}(t) = 
    (1 - \tilde{Z}_1(t) - \dots - \tilde{Z}_i(t))  Y_{i+1}(\tilde{\tau}_i(t)) \\
    \forall t \ge 0,\; \tilde{\tau}_{i+1}(t) = 
    \int_0^t \frac{1}{1 - \tilde{Z}_1(s) - \dots - \tilde{Z}_{i+1}(s)} \diff s.
\end{gather*}
We finally set $\tilde{Z}_n = 1 - \tilde{Z}_1 - \dots - \tilde{Z}_{n-1}$.

\begin{proposition}
    The process $(\tilde{Z}_1, \dots, \tilde{Z}_n)$ defined above is distributed
    as a $n$-dimensional Wright-Fisher diffusion conditioned on 
    $\Set{\zeta_n < \dots < \zeta_1}$.
\end{proposition}

\begin{proof}
    We prove the result inductively. For $n = 2$, conditioning $(X_1, X_2)$
    on its extinction order amounts to conditioning it on the fixation of $X_1$,
    and \Cref{lem:firstConditioning} shows that the result holds.

    Let $(Y_1, \dots, Y_{n-1})$ be the i.i.d.\ diffusions defined above. We
    first define
    \[
        \forall t \ge 0,\; \tilde{Z}'_2(t) = Y_2(t),\quad
        \forall t \ge 0,\; \tau'_2(t) = \int_0^t \frac{1}{1-\tilde{Z}'_2(s)} \diff s
    \]
    and then define inductively, for $i < n-1$,
    \begin{gather*}
        \forall t \ge 0,\; \tilde{Z}'_{i+1}(t) =
        (1-\tilde{Z}'_2(t)-\dots-\tilde{Z}'_i(t))
        Y_{i+1}(\tilde{\tau}_i'(t)), \\
        \forall t \ge 0,\; \tilde{\tau}'_{i+1}(t) = 
        \int_0^t \frac{1}{1 - \tilde{Z}'_2(s) - \dots
        -\tilde{Z}'_{i+1}(s)} \diff s,
    \end{gather*}
    and $\tilde{Z}'_n = 1 - \tilde{Z}'_2 - \dots - \tilde{Z}'_{n-1}$.
    By induction, we can suppose that $(\tilde{Z}'_2, \dots, \tilde{Z}'_n)$
    is distributed as a $(n-1)$-dimensional Wright-Fisher diffusion
    conditioned on its extinction order. We first claim that the process
    defined as 
    \begin{gather*}
        \forall t \ge 0,\; \tilde{Z}_1(t) = Y_1(t),\\
        \forall i > 1, \forall t \ge 0,\; \tilde{Z}_i(t) = (1-\tilde{Z}_1(t))
        \tilde{Z}'_i(\tilde{\tau}_1(t))
    \end{gather*}
    is distributed as a $n$-dimensional Wright-Fisher diffusion conditioned
    on its extinction order.

    To see this, let $(X_2, \dots, X_n)$ be a $(n-1)$-dimensional
    unconditioned Wright-Fisher diffusion, independent of $Y_1$, and 
    recall the definition of $(Z_1, \dots, Z_n)$ from
    \Cref{lem:firstConditioning}. Consider
    \[
        \zeta'_i = \inf \Set{t \ge 0 \suchthat Z_i(t) = 0},\quad
        \zeta_i = \inf \Set{t \ge 0 \suchthat X_i(t) = 0}
    \]
    the extinction times of $Z_i$ and $X_i$.
    Lemma~\ref{lem:firstConditioning} ensures that $(Z_1, \dots, Z_n)$
    is distributed as a Wright-Fisher diffusion conditioned on 
    the fixation of $Z_1$. Thus, the process $(Z_1, \dots, Z_n)$
    further conditioned on $\Set{\zeta'_n < \dots < \zeta'_2}$ has the 
    distribution of a Wright-Fisher diffusion conditioned on its
    extinction order. Now notice that 
    \[
        \Set{\zeta'_n < \dots < \zeta'_2} = \Set{\zeta_n < \dots < \zeta_2}.
    \]
    Thus conditioning $(Z_1, \dots, Z_n)$ on $\Set{\zeta'_n < \dots
    \zeta'_2}$ amounts to conditioning $(X_2, \dots, X_n)$ on 
    $\Set{\zeta_n < \dots < \zeta_2}$, that is, conditioning it
    on its fixation order. As $\Set{\zeta_n < \dots < \zeta_2}$ is
    independent of $Z_1$, conditioning the process $(Z_1, \dots, Z_n)$
    on this event is equivalent to replacing $(X_2, \dots, X_n)$ by
    $(\tilde{Z}'_2, \dots, \tilde{Z}'_n)$ in the construction of $(Z_1,
    \dots, Z_n)$, and this proves the claim.

    \medskip

    It only remains to show that $\tilde{Z}_{i+1}$ as defined in the proof can
    be written
    \[
        \forall i > 1,\; \tilde{Z}_{i+1}(t) = (1 - \tilde{Z}_1(t) - \dots -
        \tilde{Z}_i(t)) Y_i(\tilde{\tau}_i(t)).
    \]
    A direct calculation first shows that
    \[
        \forall i > 1,\; \forall t \ge 0,\; \tilde{\tau}_i(t) =
        \tilde{\tau}'_i(\tilde{\tau}_1(t))
    \]
    and the result follows.
\end{proof}

We end this section by pointing out the following fact
that will be required in the next section.
We have only defined the Wright-Fisher diffusion conditioned on
its extinction order for an initial condition $(x_1, \dots, x_n)$
such that for all $1 \le i \le n$, $x_i > 0$. Nevertheless, the processes 
$Y_i$ have an entrance boundary at $0$. Thus there exists a unique extension
of the process $(Y_1, \dots, Y_{n-1})$ started from $(0, \dots, 0)$
that remains Feller, see e.g.~\cite{kallenberg_foundations_2002},
Chapter~23. This
shows that a Wright-Fisher diffusion conditioned on its fixation order
$(\tilde{Z}_1, \dots, \tilde{Z}_n)$ admits a Feller extension for the
initial condition $(0, \dots, 0, 1)$.

\subsection{Proof of \texorpdfstring{\Cref{Thm:frequencies}}{Theorem~\ref{Thm:frequencies}}}
    \label{SS:evesProof}

Let $(\rho_t)_{t \ge 0}$ be a Fleming-Viot process, and let
$(\e_i)_{i \ge 1}$ be its Eves. In this section we end the 
proof of \Cref{Thm:frequencies} by showing that the distribution
of the sequence of processes 
$(\rho_t(\Set{\e_1}), \rho_t(\Set{\e_2}), \dots;\, t \ge 0)$
is that of a Wright-Fisher diffusion conditioned on its fixation
order.

The result we want to prove is the direct extension of Theorem~4 
of~\cite{bertoin_stochastic_2003}. Reformulated in our setting, 
this theorem proves that $(\rho_t(\Set{\e_1});\, t \ge 0)$ is 
distributed as the solution to \cref{eq:hTransform}
started from $0$. We now give a similar representation for 
the process 
$(\rho_t(\Set{\e_1}), \dots, \rho_t(\Set{\e_n});\, t \ge 0)$
giving the size of the progeny of the first $n$ Eves.
\begin{proposition}
    Let $(\rho_t)_{t \ge 0}$ be a Fleming-Viot process, and 
    $(\e_i)_{i \ge 1}$ be its Eves. Then for any $n \ge 1$,
    the process $(\rho_t(\Set{\e_1}), \dots, \rho_t(\Set{\e_n});\, t \ge 0)$
    is distributed as $(\tilde{Z}_1, \dots, \tilde{Z}_n)$ where
    $(\tilde{Z}_1, \dots, \tilde{Z}_{n+1})$ is a $(n+1)$-dimensional
    Wright-Fisher diffusion conditioned on its extinction order,
    started from $(0, \dots, 0, 1)$.
\end{proposition}
\begin{proof}
    We realize a similar computation as in the proof of Theorem~4
    of~\cite{bertoin_stochastic_2003}. The proof requires three facts.
    First notice that
    \[
        \lim_{m \to \infty} \rho_t\big(\OCInterval[\big]{\textstyle\frac{\Floor{m \e_i}}{m}, 
        \frac{\Floor{m \e_i + 1}}{m}}\big) = \rho_t(\Set{\e_i}).
    \]
    Then, if $I_1, \dots, I_n$ are $n$ disjoint intervals of length
    $1/m$, due to exchangeability of the increments of 
    bridges, the process $(\rho_t(I_1), \dots, \rho_t(I_n);\, t \ge 0)$ is 
    distributed as the process
    \[
        \textstyle\Big(\rho_t\big(\OCInterval{0, \frac{1}{m}}\big), \dots,
        \rho_t\big(\OCInterval{\frac{n-1}{m}, \frac{n}{m}}\big);\, t \ge 0\Big)
    \]
    which is distributed as the $n$ first coordinates of a
    $(n+1)$-dimensional Wright-Fisher diffusion started from 
    $(\tfrac{1}{m}, \dots, \tfrac{1}{m}, 1-\tfrac{n}{m})$.
    
    Finally, notice that on the event $\Set{\forall i \ne j \in \Set{1, \dots, n},\; \Floor{m \e_i}
    \ne \Floor{m \e_j}}$, conditioning the process
    \[
        \textstyle\Big(\rho_t\big(\OCInterval{0, \frac{1}{m}}\big), \dots,
        \rho_t\big(\OCInterval{\frac{n-1}{m}, \frac{n}{m}}\big);\, t \ge 0\Big)
    \]
    on its extinction order as in \Cref{SS:WFconditioned} is
    equivalent to conditioning it on the location of the 
    Eves, i.e., on the event
    $\Set{\forall k \in \Set{1, \dots, n},\; \e_k \in
    \OCInterval{\frac{k-1}{m}, \frac{k}{m}}}$.

    We can now proceed to the calculation.
    Let $0 \le t_1 < \dots < t_p$ and let $\phi_1, \dots, \phi_p$ be
    continuous bounded functions. Consider $(\tilde{Z}_1, \dots, \tilde{Z}_{n+1})$
    a $(n+1)$-dimensional Wright-Fisher diffusion conditioned on 
    its extinction order. Then
    \begin{align*}
        &\E \Big[
                \phi_1 \big(
                    \rho_{t_1}(\Set{\e_1}), 
                    \dots, \rho_{t_1}(\Set{\e_n}) \big) 
                \dots 
                \phi_p \big(
                    \rho_{t_p}(\Set{\e_1}), 
                    \dots, \rho_{t_p}(\Set{\e_n}) \big) 
            \Big]\\
        &= \lim_{m \to \infty} \sum_{i_1=0}^{m-1} \dots \sum_{i_n=0}^{m-1}
        \textstyle \E \Big[
            \phi_1 \big(
                \rho_{t_1}\big(\OCInterval{\frac{i_1}{m}, \frac{i_1+1}{m}}\big), 
                \dots, 
                \rho_{t_1}\big(\OCInterval{\frac{i_n}{m}, \frac{i_n+1}{m}}\big)
                \big)
            \dots \\
            &\textstyle\qquad \qquad\qquad\qquad\quad\;\;\,
            \phi_p \big(
                \rho_{t_p}\big(\OCInterval{\frac{i_1}{m}, \frac{i_1+1}{m}}\big), 
                \dots, 
                \rho_{t_p}\big(\OCInterval{\frac{i_n}{m}, \frac{i_n+1}{m}}\big)
                \big)
            \Indic{\forall k \in \Set{1, \dots, n},\; \e_k \in
                   \OCInterval{\frac{i_k}{m}, \frac{i_k+1}{m}}}
        \Big]\\
        &= \lim_{m \to \infty} m^n
        \textstyle \E \Big[
            \phi_1\big(\rho_{t_1}\big(\OCInterval{0, \frac{1}{m}}\big), \dots, 
                   \rho_{t_1}\big(\OCInterval{\frac{n-1}{m}, \frac{n}{m}}\big)\big) 
            \dots\\
            &\textstyle\quad\qquad\qquad\quad
            \phi_p\big(\rho_{t_p}\big(\OCInterval{0, \frac{1}{m}}\big), \dots, 
                   \rho_{t_p}\big(\OCInterval{\frac{n-1}{m}, \frac{n}{m}}\big)\big) 
            \Indic{\forall k \in \Set{1, \dots, n},\; \e_k \in
                \OCInterval{\frac{k-1}{m}, \frac{k}{m}}}
        \Big]\\
        &= \lim_{m \to \infty} 
        \textstyle \E \Big[
            \phi_1\big(\tilde{Z}_1(t_1), \dots, \tilde{Z}_n(t_1)\big) 
            \dots
            \phi_p\big(\tilde{Z}_1(t_p), \dots, \tilde{Z}_n(t_p)\big) 
            \;|\; \tilde{Z}_1(0) = \dots = \tilde{Z}_n(0) = \frac{1}{m}
        \Big]\\
        &= 
        \E \Big[
            \phi_1\big(\tilde{Z}_1(t_1), \dots, \tilde{Z}_n(t_1)\big) 
            \dots
            \phi_p\big(\tilde{Z}_1(t_p), \dots, \tilde{Z}_n(t_p)\big) 
            \;|\; \tilde{Z}_1(0) = \dots = \tilde{Z}_n(0) = 0
        \Big],
    \end{align*}
    where, the last line comes from the Feller property of the process
    $(\tilde{Z}_1, \dots, \tilde{Z}_{n+1})$.
\end{proof}

Our current proof of \Cref{Thm:frequencies} relies on calculations
specific to the Wright-Fisher diffusion. We end this section by
discussing a potential alternative proof of this result that would more
easily generalize to Beta-coalescents. 

The Feller branching diffusion describes the size of a population 
where different individuals die and reproduce independently. Similarly to the 
Fleming-Viot process, it is possible to define a measure-valued process,
called the Dawson-Watanabe process, that encodes the size of the
offspring of each individual in the initial population, 
see e.g.~\cite{etheridge_2000}. (Note that there are no mutations here,
i.e., no spatial motion of the particles.) Its total mass is then
distributed as a Feller diffusion. Starting from a Dawson-Watanabe
process, one can renormalize it by its total mass to obtain a process
valued in the space of probability measures. Then the resulting renormalized
process is distributed as a time-changed Fleming-Viot process, 
see~\cite{birkner_alpha_2005}.

Let us now discuss the results of \Cref{SS:WFconditioned} in the light of
this new construction. The key point of \Cref{SS:WFconditioned} is that 
after removing one family from a Fleming-Viot process and renormalizing
the remainder of the population to have mass one, the resulting process
remains distributed as an independent time-changed Fleming-Viot process.
Suppose that the Fleming-Viot process has been obtained by renormalizing
a Dawson-Watanabe process. Then removing a family from the Fleming-Viot
process amounts to removing a family from the original Dawson-Watanabe
process. By the branching property, removing this family does not change
the distribution of the remainder of the population, which remains
distributed as an independent Dawson-Watanabe process. Thus when
renormalizing the remainder of the population to have size one, we obtain
a new time-changed Fleming-Viot process, independent of the removed
family. In other words, the results of \Cref{SS:WFconditioned}
essentially originate from the fact that the Fleming-Viot process can be
seen as a renormalized branching measure-valued process.

A similar link has been obtained in~\cite{birkner_alpha_2005} between 
the $\Lambda$-Fleming-Viot processes associated to Beta-coalescents and a
family of $\alpha$-stable measure-valued branching processes. Thus we believe
that one could derive a similar, but less explicit, representation of the
asymptotic frequencies of the stationary distribution of the
Beta-coalescents with erosion than the one obtained in \Cref{Thm:frequencies}.

\bibliography{kingman_erosion}

\appendix

\section{Proof of \texorpdfstring{\Cref{prop:ancestralMarkov}}{Proposition~\ref{prop:ancestralMarkov}}} \label{A:counting}

In this section, we prove that the ancestral process of Kingman's
coalescent with immigration is Markovian. To do this, consider a version
of Kingman's coalescent with immigration $(\bar{\Pi}_t)_{t \in \R}$, and
let $(\bar\Pi_i)_{i \in \Z}$ be its embedded chain, i.e., the sequence of
states visited by $(\bar{\Pi}_t)_{t \in \R}$, where $\bar{\Pi}_0$ is the
state at time $t = 0$.  We count the number of trajectories of
$(\bar{\Pi}_i)_{i \in \Z}$ that produce a given trajectory of $(\Ai_i)_{i \ge 0}$, 
the embedded chain of $(\Ai_t)_{t \ge 0}$.

First, note that given the values of $(\bar{\Pi}_{-n}, \dots, \bar{\Pi}_0)$ and a
uniform permutation $\sigma$ of the blocks of $\bar{\Pi}_0$, one can uniquely
reconstruct the values of $(\Ai_0, \dots, \Ai_n)$. We now fix a sequence
$(a_0, \dots, a_n)$ of possible values of $(\Ai_0, \dots, \Ai_n)$, and
a partition $\bar\pi_{-n}$ with $\abs{a_n}$ blocks, where $\abs{a_n}$
is the total number of particles of $a_n$. Our first task is to 
count the number of trajectories of $(\bar{\Pi}_{-n}, \dots, \bar{\Pi}_0)$ starting
from $\bar{\pi}_{-n}$, and of labelings $\sigma$ of the blocks of $\bar{\Pi}_0$
such that $(\Ai_0, \dots, \Ai_n) = (a_0, \dots, a_n)$. Before stating
the result we need to introduce one notation. The variable $\Ai_{k+1}$
is obtained from $\Ai_k$ by splitting or killing one particle. Let us
denote $\ell_k$ the label of this particle. That is, $\ell_k$ is the 
unique integer such that
\[
    \abs{\Ai_{k+1}(\ell_k) - \Ai_{k}(\ell_k)} = 1, \quad
    \forall i \ne \ell_k,\; \abs{\Ai_{k+1}(i) - \Ai_{k}(i)} = 0.
\]

\begin{lemma} \label{lem:couting}
    Fix a sequence of states $(a_0, \dots, a_n)$ of $(\Ai_0, \dots,
    \Ai_n)$, and a partition $\bar{\pi}_{-n}$ of $\Set{i \in \Z \suchthat
    i \le -n}$ with $\Abs{a_n}$ blocks. Then the number of trajectories of
    $(\bar{\Pi}_{-n}, \dots, \bar{\Pi}_0)$ and labelings of the blocks of
    $\bar{\Pi}_0$ such that $(\Ai_0, \dots, \Ai_n) = (a_0, \dots, a_n)$
    and $\bar{\Pi}_{-n} = \bar{\pi}_{-n}$ is
    \[
        \frac{\Abs{a_n}!}{2^b} a_0(\ell_0) \dots a_{n-1}(\ell_{n-1}),
    \]
    where $b$ is the number of birth events along the sequence
    $(a_0, \dots, a_n)$.
\end{lemma}

\begin{proof}
    Each trajectory of $(\bar{\Pi}_{-n}, \dots, \bar{\Pi}_0)$ naturally
    encodes a forest that can be built as follows. Choose any labeling
    of the blocks of $\bar{\Pi}_{-n}$, and for each block add an initial leaf
    with the corresponding label. Suppose that the forest corresponding to 
    $(\bar{\Pi}_{-n}, \dots, \bar{\Pi}_{-k})$ has been built. If
    $\bar{\Pi}_{-k+1}$ is obtained from $\bar{\Pi}_{-k}$ by immigrating
    a new particle, then add a new isolated vertex. Otherwise, a
    coalescence event has occurred between two blocks of
    $\bar{\Pi}_{-k}$. Then add a new internal node and connect it to the
    nodes corresponding to the two blocks that have coalesced.
    Once the forest representing $(\bar{\Pi}_{-n}, \dots, \bar{\Pi}_0)$
    is built, by construction the nodes corresponding to $\bar{\Pi}_0$
    all belong to different trees. We set them to be the roots of their
    respective trees, and label them according to the partition 
    $\sigma$. (Notice that the resulting forest is endowed with some
    additional structure: the nodes added along the procedure are totally
    ordered by the induction step at which they have been added.)

    Counting trajectories of $(\bar{\Pi}_{-n}, \dots, \bar{\Pi}_0)$
    now amounts to counting forests. Instead of building the forests by
    starting from the leaves as above, we build a forest with ancestral
    sequence $(a_0, \dots, a_n)$ by starting from the roots. Initially,
    consider a set of $\Abs{a_0}$ roots, labeled by $\Set{1, \dots, \Abs{a_0}}$,
    that represent the particles of $a_0$. Nodes can be in two states:
    active or inactive. Active nodes represent the particles that are
    still alive in the population while inactive nodes represent the 
    dead particles. Initialy all roots are active. We build the forest 
    recursively. Suppose that at step $k$ we have built a forest such
    that for all $i$ there are $a_k(i)$ nodes that are active in the 
    tree with root $i$. If a particle with label $\ell_k$ has died from
    $a_k$ to $a_{k+1}$, we inactivate one of the nodes belonging to the 
    tree with root $\ell_k$. There are $a_k(\ell_k)$ such nodes.
    Similarly, if a particle has split from $a_k$ to $a_{k+1}$, we
    inactivate one node in the tree $\ell_k$, and connect it to two new
    active nodes. There are again $a_k(\ell_k)$ active nodes in the tree
    $\ell_k$. After step $n$, we have built a forest with ancestral
    sequence $(a_0, \dots, a_n)$. We assign the blocks of $\bar{\Pi}_{-n}$
    to the remaining active nodes of the forest by choosing one of
    the $\Abs{a_n}!$ permutations of the blocks.

    There are 
    \[
        \Abs{a_n}!\, a_0(\ell_0) \dots a_{n-1}(\ell_{n-1})
    \]
    outputs of the previous construction, and all forests
    with ancestral sequence $(a_0, \dots, a_n)$ can be obtained that way.
    However, due to symmetries, some forests can be obtained multiple
    times through this construction. More precisely, at each birth
    events, the two daughter nodes are indistinguishable. Interchanging
    the trees corresponding to the offspring of these two nodes yields
    the same forest. Thus, the actual number of forests with ancestral
    sequence $(a_0, \dots, a_n)$ is 
    \[
        \frac{\Abs{a_n}!}{2^b} a_0(\ell_0) \dots a_{n-1}(\ell_{n-1})
    \]
    where $b$ is the number of birth events, and the result is proved.
\end{proof}

\begin{lemma} \label{lem:reversible}
    Let $(M_t)_{t \in \R}$ be the process counting the number of 
    blocks of Kingman's coalescent with immigration. Then
    $(M_t)_{t \in \R}$ is a reversible process.
\end{lemma}

\begin{proof}
    Let us compute the stationary distribution of $(M_t)_{t \in \R}$.
    As $(M_t)_{t \in \R}$ jumps from $k$ to $k+1$ at rate $d$ and 
    from $k$ to $k-1$ at rate $k(k-1)/2$, a usual calculation
    shows that its stationary distribution $(\nu_k)_{k \ge 1}$
    is
    \[
        \forall k \ge 1,\; \nu_k \propto \frac{(2d)^k}{k!\,(k-1)!}
    \]
    where the renormalization constant is obtained by summing over
    all the terms.
    Thus a direct calculation now proves that $(\nu_k)_{k \ge 1}$
    fulfills the detailed balance equation
    \[
        \forall k \ge 1,\; d \nu_k = \frac{k(k+1)}{2} \nu_{k+1}
    \]
    and thus that $(M_t)_{t \in \R}$ is reversible.
\end{proof}

We are now ready to prove \Cref{prop:ancestralMarkov}

\begin{proof}[Proof of \Cref{prop:ancestralMarkov}]
    Recall the notations from \Cref{SS:ancestral}. As proved in
    Lemma~\ref{lem:reversible}, the 
    process $(M_t)_{t \in \R}$ that counts the number of the blocks
    of Kingman's coalescent with immigration is a reversible 
    Markov process. Thus, the process $(N_t)_{t \ge 0}$ that gives the 
    number of particles of $(\Ai_t)_{t \ge 0}$ is a stationary process
    jumping from $k$ to $k+1$ at rate $d$, and from $k$ to $k-1$ at
    rate $k(k-1)/2$. Hence, the result is proved if we show that 
    conditional on $(N_0, \dots, N_n)$, the type of the particle that
    dies or splits from $\Ai_k$ to $\Ai_{k+1}$ is chosen with a
    probability proportional to the vector $\Ai_k$. We have
    \begin{align*}
        \Prob{\Ai_0 = a_0, \dots, \Ai_n = a_n}
        = \sum_{(\bar{\pi}_{-n}, \dots, \bar{\pi}_{0})} \sum_s
        \Prob{\forall i < n,\; \bar{\Pi}_{-i} = \bar{\pi}_{-i},\; \sigma = s \given \bar{\Pi}_{-n} = \bar{\pi}_{-n}}
        \Prob{\bar{\Pi}_{-n} = \bar{\pi}_{-n}}
    \end{align*}
    where the sum is taken over all partitions $\bar{\pi}_{-n}$ of 
    $\Set{i \in \Z \suchthat i \le -n}$ with $\abs{a_0}$ blocks, 
    all trajectories $(\bar{\pi}_{-n+1}, \dots, \bar{\pi}_0)$ and
    labelings $s$ of the blocks of $\bar{\pi}_0$ such that 
    $(\Ai_0, \dots, \Ai_n) = (a_0, \dots, a_n)$. Now notice that 
    the probability of seeing such a trajectory and labeling does
    only depend on the sequence of number of blocks $(\abs{a_0}, \dots,
    \abs{a_n})$. Indeed we have
    \begin{align*}
        \Prob{\forall i < n,\; \bar{\Pi}_{-i} = \bar{\pi}_{-i},\; \sigma = s \given \bar{\Pi}_{-n} = \bar{\pi}_{-n}} 
         = \frac{1}{\abs{a_0}!} \prod_{i = 0}^{n-1}
         \frac{d \Indic{\abs{a_{i+1}} - \abs{a_i} = -1} +
             \Indic{\abs{a_{i+1}}-\abs{a_i} = 1}}{d +
         \abs{a_{i+1}}(\abs{a_{i+1}}-1)/2}.
    \end{align*}
    Thus the probability of the event $\Set{\Ai_0 = a_0, \dots, \Ai_n=a_n}$
    is proportional to the number of terms in the sum, and thus 
    to the number of trajectories of $(\bar{\Pi}_{-n}, \dots,
    \bar{\Pi}_0)$ that correspond to this ancestral sequence.
    Hence, Lemma~\ref{lem:couting} shows that 
    \[
        \Prob{\Ai_0 = a_0, \dots, \Ai_n = a_n} \propto a_0(\ell_0) \dots
        a_{n-1}(\ell_{n-1}),
    \]
    where the coefficient only depends on $(\abs{a_0}, \dots,
    \abs{a_n})$. This proves the result.
\end{proof}

Let us end this section by discussing a possible extension to 
$\Lambda$-coalescents. The key point here is that conditionally
on the block counting process, the particles that die or split
are chosen uniformly in the population. This is a consequence of
1) \Cref{lem:couting} and 2) the fact that all trajectories with
a given sequence of number of blocks have the same probability.
The second point is a 
consequence of exchangeability so remains valid for $\Lambda$-coalescents. As for \Cref{lem:couting},
the proof could be easily adapted to $\Lambda$-coalescents
with immigration. (The factor $2^b$ should be replaced by
the product of the number of blocks involved in coalescence
events.)

Thus, the only difference between Kingman's coalescent with 
immigration and more general $\Lambda$-coalescents with immigration
is that the block counting process is no longer reversible.
Hence we cannot obtain a closed form for the transition rates
of the corresponding ancestral processes. Nevertheless, we believe
that in some cases it should be possible to obtain a result similar
to \Cref{prop:sizeErosion} by using the same techniques as
in this paper, if one can derive a good enough approximation for 
the stationary distribution of the number of blocks.

\end{document}